\documentclass[hidelinks,onefignum,onetabnum]{siamart251216}
\usepackage{lipsum}
\usepackage{amsfonts}
\usepackage{graphicx}
\usepackage{epstopdf}
\usepackage{algpseudocode}
\usepackage{enumitem}
\usepackage[nameinlink,capitalize]{cleveref}
\ifpdf
  \DeclareGraphicsExtensions{.eps,.pdf,.png,.jpg}
\else
  \DeclareGraphicsExtensions{.eps}
\fi

\usepackage{amsfonts}
\usepackage{amsmath,amssymb,amsfonts,mathtools,bm}
\usepackage{mathrsfs}
\usepackage{bbm}
\usepackage{multirow}
\usepackage{array}
\usepackage{caption}
\usepackage{subcaption}
\usepackage{float}
\usepackage{enumitem}

\newcommand{\R}{\mathbb{R}}

\newcommand{\I}{\bm{I}}
\newcommand{\W}{\bm{W}}
\newcommand{\V}{\bm{V}}
\newcommand{\Z}{\bm{Z}}
\newcommand{\U}{\bm{U}}
\newcommand{\A}{\bm{A}}
\newcommand{\B}{\bm{B}}
\newcommand{\s}{\bm{s}}
\newcommand{\Hessian}{\bm{H}_f}
\newcommand{\approxH}{\tilde{\bm{H}}}
\newcommand{\param}{\bm{\theta}} 
\newcommand{\stepsize}{\Delta t}
\newcommand{\indexopt}{n}
\newcommand{\recomputefactor}{\kappa}

\newcommand{\paramdiff}{\bm{\theta}_{\indexopt+1}-\bm{\theta}_\indexopt}
\newcommand{\operator}{\mathcal{L}} 
\newcommand{\tilr}{\tilde{r}}

\newcommand{\approxrank}{m}

\newsiamremark{remark}{Remark}
\newsiamremark{hypothesis}{Hypothesis}
\crefname{hypothesis}{Hypothesis}{Hypotheses}
\newsiamthm{claim}{Claim}
\newsiamremark{fact}{Fact}
\crefname{fact}{Fact}{Facts}

\headers{
Nystr\"om-Enhanced RSAV}
{Ryo Sagawa, Daisuke Furihata, Yuto Miyatake}

\title{Accelerating SAV-based optimization via randomized low-rank Hessian approximation
}

\author{
Ryo Sagawa\thanks{Department of Pure and Applied Mathematics, Graduate School of Information Science and Technology, The University of Osaka
  (\email{sagawa@cas.cmc.osaka-u.ac.jp}).}
\and Daisuke Furihata\thanks{D3 Center, The University of Osaka
  (\email{daisuke.furihata.cmc@osaka-u.ac.jp}).}
\and Yuto Miyatake\thanks{D3 Center, The University of Osaka
  (\email{yuto.miyatake.cmc@osaka-u.ac.jp}).}
  }

\usepackage{amsopn}

\ifpdf
\hypersetup{
  pdftitle={
  Accelerating SAV-based optimization via randomized low-rank Hessian approximation
  },
  pdfauthor={Ryo Sagawa, Daisuke Furihata, and Yuto Miyatake}
}
\fi

\begin{document}

\maketitle

\begin{abstract}
We propose a new optimization method, the Nystr\"om-enhanced relaxed scalar auxiliary variable method (N-RSAV), which incorporates curvature information into the RSAV framework to accelerate convergence while preserving an unconditional modified energy dissipation law. 
Existing RSAV-based methods rely solely on first-order information and often suffer from slow convergence, particularly for ill-conditioned problems such as those arising in physics-informed neural networks (PINNs).
To address this limitation, 
we design the linear operator in the RSAV scheme using approximate Hessian information obtained from a randomized low-rank Nyström approximation.
To preserve the dissipation structure, we enforce positive semidefiniteness through eigenvalue truncation. 
Furthermore, we introduce an adaptive strategy that reuses the approximate Hessian based on the deviation between the original and modified energies, significantly reducing computational cost.
We also provide a convergence analysis of the RSAV scheme with a general positive semidefinite operator under the Polyak--\L{}ojasiewicz (PL) condition and establish corresponding convergence guarantees for N-RSAV under the PL condition and an additional convexity assumption.
Numerical experiments on ill-conditioned problems with effectively low-rank structure, including convex quadratic problems and training of PINNs, demonstrate that the proposed methods achieve substantially faster convergence than conventional RSAV-based approaches.
\end{abstract}



\section{Introduction}
\label{sec:introduction}
We consider the unconstrained optimization problem
\begin{equation}
    \min_{\param \in \R^d} f(\param),
    \label{eq:opt_problem}
\end{equation}
where $f: \R^d \to \R$ is assumed to be sufficiently smooth and to admit a global minimizer $\param^\ast$. 
Such unconstrained optimization problems naturally arise in various fields, including machine learning~\cite{ml_opt} and inverse problems~\cite{inverse_problem_opt}. 
A fundamental approach to solving~\eqref{eq:opt_problem} is gradient descent, which iteratively updates the parameter $\param_\indexopt$ as
\begin{equation}
\param_{\indexopt+1} = \param_\indexopt - \stepsize \nabla f(\param_\indexopt)  \quad (\indexopt \geq 0), \notag
\end{equation}
where $\stepsize > 0$ is the step size. 
This scheme can be interpreted as the explicit Euler discretization of the gradient flow:
\begin{equation}
\frac{d\param(t)}{dt} = -\nabla f(\param(t)). \label{eq:gradient_flow}
\end{equation}
The solution to the gradient flow~\eqref{eq:gradient_flow} satisfies the dissipation law
\begin{equation}
\frac{d}{dt} f(\param(t)) \leq 0, \notag
\end{equation}
which ensures that the objective values decrease monotonically over time. 
However, in gradient descent, the dissipation property is generally lost if the step size $\stepsize$ is chosen too large, which can lead to instability. 
Therefore, selecting an appropriate step size that balances convergence speed and stability has long been an important issue. Classical approaches include backtracking line search based on the Armijo condition~\cite{back}, while more recent developments include adaptive methods such as RMSprop~\cite{rmsprop}.

In this paper, we consider the scalar auxiliary variable (SAV) approach~\cite{sav}.
Originally developed as a structure-preserving discretization method for gradient flows, the SAV framework has recently been applied to optimization problems~\cite{sav_opt,sav_opt2,sav_opt3}.
By introducing an auxiliary variable $r_\indexopt$, the SAV scheme enforces the dissipation of a modified energy $r_\indexopt^2$, rather than the original energy $f(\param_\indexopt)$.
In general, the method permits relatively large step sizes, which may improve convergence performance.
Moreover, the introduction of the auxiliary variable makes the resulting iteration scheme either explicit or linearly implicit, thereby avoiding fully nonlinear updates.

However, existing SAV-based optimization methods rely only on first-order information, namely the gradient of the objective function, and thus their convergence speed may still be limited. 
In particular, for problems where the Hessian of the objective function is ill-conditioned, such as those arising in physics-informed neural networks (PINNs)~\cite{pinns}, gradient-based methods often fail to capture curvature information adequately, resulting in very slow convergence~\cite{pinns_second_order_opt1,pinns_second_order_opt2}.

To overcome this limitation, we propose a new optimization method that incorporates approximate Hessian information into the SAV framework. 
Specifically, we construct an approximate Hessian using the Nyström method~\cite{Nystrom}, a randomized low-rank approximation technique, together with eigenvalue truncation to ensure positive semidefiniteness. This design enables us to exploit local curvature information near a minimizer while maintaining both computational efficiency and the modified dissipation property of the SAV scheme.
To further reduce the computational cost, we also introduce an adaptive strategy that reuses the approximate Hessian based on the deviation between the original and modified energies.

The remainder of this paper is organized as follows. Section~\ref{sec:preliminaries} reviews optimization methods based on the SAV approach and its relaxed variant, the RSAV approach, and discusses the limitations of the RSAV scheme.  Section~\ref{sec:proposed_method} presents the proposed N-RSAV method and its adaptive variant, AN-RSAV, and provides remarks on the proposed methods. 
Section~\ref{sec:convergence_analysis} analyzes the convergence of RSAV with a general positive semidefinite operator $\operator_\indexopt$ under the PL condition, and then applies the result to the proposed N-RSAV method under the PL condition and an additional convexity assumption.
Section~\ref{sec:numerics}  demonstrates the effectiveness of the proposed methods through two numerical experiments: optimization of a quadratic function and training of PINNs, followed by conclusions in Section~\ref{sec:conclusion}.

Throughout the paper, $\|\cdot \|$ denotes the L2 norm for vectors and the spectral norm for matrices. 
For a positive definite matrix $\A \in \R^{d \times d}$, we define the $A$-weighted norm: $\|\param\|_{\A} \coloneqq \sqrt{\param^\top \A \param}  \quad (\param \in \R^d)$. 
The identity matrix in $\R^{d \times d}$ is denoted by $\I^{(d)}\in \R^{d \times d}$. 
For symmetric matrices $\A,\B \in \R^{d\times d}$, 
\(\A \preceq \B\) denotes the Loewner order, meaning that $\B-\A$ is positive semidefinite.
Equivalently, \(\A \preceq \B\) holds if and only if 
$
\param^\top \A\param \le \param^\top \B\param$ for all $\param \in \R^d$.
In particular, $\A \succeq \bm{0}$ indicates that $\A$ is positive semidefinite.

\section{Preliminaries: SAV/RSAV-based optimization methods}
\label{sec:preliminaries}
In this section, we review the SAV- and RSAV-based  optimization methods~\cite{sav_opt}.

\subsection{SAV/ RSAV scheme}
We introduce the scalar auxiliary variable:
\begin{equation}
    r(t) := \sqrt{f(\param(t)) + C}, 
     \notag
\end{equation}
where $C$ is a constant such that $C > -\min_{\param \in \R^d} f(\param)$. We then split the objective function into two terms: 
\begin{equation}
    f(\param(t)) = \frac{1}{2}\param(t)^\top \operator(t)\param(t) + \left[ f(\param(t)) - \frac{1}{2}\param(t)^\top \operator(t)\param(t) \right],
   \notag
\end{equation}
where $\operator(t) : \R^d \rightarrow \R^d$ is an arbitrary symmetric positive semidefinite linear operator, which might also depend on $\param(t)$.
Then the gradient flow~\eqref{eq:gradient_flow} can be reformulated as the following extended system:
\begin{equation}
\begin{cases}
    \dfrac{dr}{dt} = \dfrac{1}{2\sqrt{f(\param(t))+C}}\nabla f(\param(t))^\top \dfrac{d\param}{dt}, \\
    \dfrac{d\param}{dt} = -\operator(t)\param(t) - \dfrac{r(t)}{\sqrt{f(\param(t))+C}}\nabla f(\param(t)) + \operator(t)\param(t).
\end{cases}
\label{eq:extended_system}
\end{equation}
A standard SAV scheme is obtained by discretizing the extended system~\eqref{eq:extended_system} with a step size $\stepsize > 0$:
\begin{equation}
 \left\{ \ 
 \begin{aligned}
    &\frac{r_{\indexopt+1}- r_\indexopt}{\stepsize} = \frac{1}{2\sqrt{f(\param_\indexopt)+C}}\nabla f(\param_\indexopt)^\top \frac{\param_{\indexopt+1}-\param_\indexopt}{\stepsize},\\
    &\frac{\param_{\indexopt+1}-\param_\indexopt }{\stepsize}= -\operator_\indexopt \param_{\indexopt+1}- \frac{r_{\indexopt+1}}{\sqrt{f(\param_\indexopt)+C}}\nabla f(\param_\indexopt)+\operator_\indexopt \param_\indexopt,
\end{aligned}
 \right.   
\label{eq:sav_discrete_scheme}
\end{equation}
where $\operator_\indexopt$ is the matrix representation of $\operator(t_n)$ and $t_\indexopt$ denotes the time associated with the approximation $\param_\indexopt$.
Equivalently, the scheme~\eqref{eq:sav_discrete_scheme} can be written as
\begin{equation}
 \left\{ \ 
 \begin{aligned}
    &r_{\indexopt+1} = \left(1 + \stepsize \dfrac{\nabla f(\param_\indexopt)^\top \bm{A}_\indexopt^{-1} \nabla f(\param_\indexopt)}{2(f(\param_\indexopt)+C)}\right)^{-1} r_\indexopt,\\
    &\param_{\indexopt+1} = \param_\indexopt - \stepsize \dfrac{ r_{\indexopt+1}}{\sqrt{f(\param_\indexopt)+C}} \A_\indexopt^{-1}\nabla f(\param_\indexopt),
\end{aligned}
 \right.   
\label{eq:sav_scheme}
\end{equation}
where $\bm{A}_\indexopt \coloneqq \I^{(d)} + \stepsize \operator_\indexopt $.
In the SAV scheme~\eqref{eq:sav_scheme}, the search direction is given by $\A_\indexopt^{-1}\nabla f(\param_\indexopt)$, and thus, it is typically necessary to solve a linear system $\A_\indexopt\bm{x}=\nabla f(\param_\indexopt)$ at each iteration. As $\bm{A}_\indexopt^{-1}$ is positive definite and $f(\param_\indexopt)+C>0$, the first equation in the scheme \eqref{eq:sav_scheme} implies that $0< r_{n+1} \leq r_n$.
The positivity follows provided that $r_0>0$, which is typically satisfied since one usually chooses $r_0 = \sqrt{f(\param_0)+C}$.

\begin{theorem}[\cite{sav_opt, sav}]
If $\operator_\indexopt$ is positive semidefinite, the SAV scheme~\eqref{eq:sav_scheme} satisfies a modified dissipation law: 
\begin{equation}
     r_{\indexopt+1}^2 -r_\indexopt^2  \le 0. \notag
\end{equation}
\end{theorem}

\begin{remark}
    For the continuous equation \eqref{eq:extended_system}, the relation $r(t)^2 = f(\param (t))+C$ holds for all $t\geq 0$.
    However, in the discrete setting, even if the initial value is chosen as $r_0 = \sqrt{f(\param_0)+C}$, the quantities $r_n^2$ and $f(\param_\indexopt)+C$ deviate from each other during the iterations of the SAV scheme.

    In what follows, we refer to $f(\param_\indexopt)$ as the \emph{original} energy and to $r_\indexopt^2$ as the \emph{modified} energy.
\end{remark}

To mitigate this discrepancy between $r_n$ and $\sqrt{f(\param_\indexopt)+C}$, a relaxed SAV (RSAV) scheme was introduced by incorporating a relaxation step~\cite{rsav_1, rsav_2}.
The RSAV scheme is given by
\begin{equation} 
    \left\{ \ 
 \begin{aligned}
     &  \tilde{r}_{\indexopt+1} = \left({1 + \stepsize \frac{\nabla f(\param_\indexopt)^{\top}\bm{A}_\indexopt^{-1}\nabla f (\param_\indexopt)}{2(f(\param_\indexopt)+C)}}\right)^{-1}r_\indexopt, \\
    &\param_{n+1} =\param_\indexopt - \stepsize \frac{\tilde{r}_{\indexopt+1}}{\sqrt{f(\param_\indexopt) +C}}\bm{A}_n^{-1}\nabla f(\param_\indexopt), \\
    & r_{\indexopt+1} = \xi_\indexopt \tilde{r}_{\indexopt+1} + (1-\xi_\indexopt)\sqrt{f(\param_{\indexopt+1})+C},
 \end{aligned}
 \right.    \label{eq:rsav_scheme}
\end{equation}
where the relaxation parameter $\xi_\indexopt \in [0,1]$ is chosen to control the discrepancy between the modified and original energies.

The relaxation scheme has the following properties.

First, under a mild assumption on $\xi_\indexopt$, the modified dissipation property still holds.

\begin{theorem}[\cite{sav_opt,rsav_1}]
Suppose that 
$\xi_\indexopt$ satisfies the condition:
\begin{equation}
r_{\indexopt+1}^2 - \tilr_{\indexopt+1}^2 - (\tilr_{\indexopt+1}- r_\indexopt)^2 \le \eta G_\indexopt,
\label{eq:cond}
\end{equation}
where
\begin{equation*}
    G_\indexopt \coloneqq \frac{1}{\stepsize}\|\paramdiff\|^2 +(\paramdiff)^\top \operator_\indexopt(\paramdiff) 
    = -2 (\tilde{r}_{\indexopt+1}-r_\indexopt)\tilde{r}_{\indexopt+1}
    \ge 0,
\end{equation*}
and $\eta \in [0,1)$ is a prescribed constant.
If $\operator_\indexopt$ is positive semidefinite, the RSAV scheme \eqref{eq:rsav_scheme} satisfies a modified dissipation law: 
\begin{equation}
     r_{\indexopt+1}^2 -r_\indexopt^2 \le - (1-\eta)G_\indexopt \le 0. \label{eq:rsav_modified_dissipation_law}
\end{equation}
\end{theorem}

Typically, $\eta$ is chosen to be close to $1$.

Second, we illustrate how the relaxation step controls the discrepancy between the modified and original energies.

Substituting the update formula of $r_{n+1}$, namely, $r_{\indexopt+1}=\xi_\indexopt \tilde{r}_{\indexopt+1} + (1-\xi_\indexopt)\sqrt{f(\param_{\indexopt+1})+C}$,
into the condition~\eqref{eq:cond},
we obtain the following quadratic inequality for
$\xi_\indexopt$:
\begin{equation}
a\xi_\indexopt^2 + b\xi_\indexopt + c \le 0.
\label{eq:cond_equivalent}
\end{equation}
Here, the coefficients are given by
\begin{equation}
 \begin{split}
        &a = \left(\tilde{r}_{\indexopt+1}-\sqrt{f(\param_{\indexopt+1})+C} \right)^2 ,\\
        &b=2 \sqrt{f(\param_{\indexopt+1})+C}\left(\tilde{r}_{\indexopt+1} - \sqrt{f(\param_{\indexopt+1})+C}\right),\\
        &c = f(\param_{\indexopt+1})+C - (\tilde{r}_{\indexopt+1})^2 +2 \eta(\tilde{r}_{\indexopt+1} - r_\indexopt)\tilde{r}_{\indexopt+1} . \notag
        \label{eq:abc}  
    \end{split}
\end{equation}

To maximize the correction effect while maintaining the constraint~\eqref{eq:cond_equivalent}, the relaxation parameter $\xi_\indexopt$ is chosen as the solution to the constrained optimization problem:
\begin{equation}
\min_{\xi \in [0,1]} \xi
\quad
\text{subject to}
\quad
a\xi^2 + b\xi + c \le 0 .
\notag
\end{equation}
This choice minimizes the weight assigned to $\tilde{r}_{\indexopt+1}$ and therefore maximizes the contribution of $\sqrt{f(\param_{\indexopt+1})+C}$, yielding the strongest correction toward the original energy.
The resulting parameter is given by
\begin{equation}
\xi_\indexopt =
\begin{cases}
\displaystyle
\max\left\{0, \frac{-b - \sqrt{b^2 - 4ac}}{2a}\right\} \quad & (a \neq 0) ,\\
0 \quad&  (a = 0).
\end{cases}
\notag
\end{equation}

\subsection{\texorpdfstring{The limitations of the RSAV scheme arising from the choice of the operator $\operator_\indexopt$}{Choice of the operator Ln}}
\label{subsec:limit_rsav}

There are many degrees of freedom in the choice of the operator $\operator_\indexopt$.
In previous studies applying RSAV to optimization problems, a simple choice,
$\operator_\indexopt = \lambda \I^{(d)}$ with a nonnegative hyperparameter $\lambda \ge 0$, was considered~\cite{sav_opt,sav_opt3}.
In this case, the update formula for $\param_\indexopt$ in the RSAV scheme \eqref{eq:rsav_scheme} simplifies to
\begin{equation*}
    \param_{\indexopt+1} = \param_\indexopt - \frac{\tilde{r}_{\indexopt+1}}{\sqrt{f(\param_\indexopt) +C}}
    \frac{\stepsize}{1 + \stepsize \lambda}
    \nabla f(\param_\indexopt).
\end{equation*}
However, this update can be viewed simply as an adaptive gradient descent method with the effective step size
\begin{equation*}
    \frac{\tilde{r}_{\indexopt+1}}{\sqrt{f(\param_\indexopt) +C}}
    \frac{\stepsize}{1 + \stepsize \lambda}.
\end{equation*}
Consequently, this approach inherits many of the well-known drawbacks of gradient descent. 
For example, when the Hessian of the objective function is ill-conditioned, the convergence rate can deteriorate significantly.

\section{Nystr\"om-enhanced RSAV method and adaptive Hessian reuse strategy}
\label{sec:proposed_method}

In this section, we propose an enhanced RSAV method that incorporates Hessian information. We also discuss a strategy for reducing the computational cost.

\subsection{Nystr\"om-enhanced RSAV method}
\label{subsec:proposed_method}

To address the limitation discussed in subsection~\ref{subsec:limit_rsav}, 
we design the operator $\operator_\indexopt$ using
Hessian information
$\Hessian(\param_\indexopt) \coloneqq \nabla^2 f(\param_\indexopt)$.
However, directly employing the exact Hessian presents two major challenges: (i) increased computational cost arising from both the construction of the Hessian and the solution of the associated linear systems $(\I^{(d)} + \stepsize  \,\Hessian(\param_\indexopt))\bm{x} = \nabla f(\param_\indexopt)$ at each iteration; and (ii) the potential loss of the modified dissipation law of the RSAV scheme when the Hessian is indefinite.
 To overcome these issues, we propose using an approximate Hessian in place of the exact Hessian.

To mitigate the first issue (i), we employ the Nyström method, a randomized low-rank approximation technique, to efficiently approximate the Hessian~\cite{Nystrom}.
The Nyström method approximates the Hessian by a matrix of rank at most $m \ll d$:
\begin{equation}
    \Hessian(\param_\indexopt) \approx \V_\indexopt \W_\indexopt^{\dagger} \V_\indexopt^\top,
    \notag
\end{equation}
where $\V_\indexopt \in \R^{d\times \approxrank}$ and $\W_\indexopt\in\R^{m\times m}$ are submatrices of the Hessian, and $\W_\indexopt^{\dagger}$ denotes the Moore--Penrose pseudoinverse of $\W_\indexopt$.
More precisely, let $[d] = \{1,2,\ldots,d\}$ be the set of column indices of the Hessian $\Hessian(\param_\indexopt) \in \mathbb{R}^{d \times d}$. 
We construct an index set $\mathcal{M} = \{\ell_1, \ell_2, \ldots, \ell_\approxrank\} \subset [d]$ by sampling $\approxrank \ll d$ indices uniformly at random.
The matrix $\V_\indexopt \in \mathbb{R}^{d \times \approxrank}$ is formed by extracting the columns of the Hessian indexed by $\mathcal{M}$:
\begin{equation}
\V_\indexopt = \left( \frac{\partial^2 f(\param_\indexopt)}{\partial \theta_i \partial \theta_{\ell_j}} \right)_{1 \le i \le d, 1 \le j \le \approxrank}. \label{eq:nystrom_sub_mat1}
\end{equation}
Similarly, the matrix $\W_n \in \mathbb{R}^{\approxrank \times \approxrank}$ is defined as the principal submatrix corresponding to $\mathcal{M}$:
\begin{equation}
\W_\indexopt = \left( \frac{\partial^2 f(\param_\indexopt)}{\partial \theta{\ell_k} \partial \theta_{\ell_j}} \right)_{1 \le k,j \le \approxrank}. \notag
\end{equation}
Note that it is not necessary to explicitly compute the full Hessian matrix to construct $\V_\indexopt$ in \eqref{eq:nystrom_sub_mat1}. 
Indeed, the $j$-th column of $\V_\indexopt$ is precisely the $\ell_j$-th column of the Hessian matrix. 
Therefore, for each $j \in \{1,2,\ldots,m\}$, it suffices to compute the Hessian-vector product
$
\Hessian(\param_\indexopt)\bm{e}_{\ell_j},$
where $\bm{e}_{\ell_j}\in \R^d$ denotes the standard basis vector whose $\ell_j$-th entry is one and whose remaining entries are zero.
Furthermore, these Hessian-vector products can be computed efficiently without explicitly computing the full Hessian~\cite{hvp,ito2019adjoint}.

To resolve the second issue (ii), we need to enforce positive semidefiniteness of the approximate Hessian.
This is achieved by performing an eigenvalue decomposition and setting all negative eigenvalues to zero. The procedure is summarized below.

We perform an eigenvalue decomposition of $\W_\indexopt$ to obtain
\begin{equation}
    \W_\indexopt = \U_\indexopt \bm{\Sigma}_\indexopt \U_\indexopt^\top. \notag
\end{equation}
Here, $\U_\indexopt \in \R^{m \times m}$ is an orthogonal matrix whose columns are the eigenvectors of $\W_\indexopt$, and $\bm{\Sigma}_\indexopt = \mathrm{diag}(\lambda_1, \ldots, \lambda_m)$ is a diagonal matrix containing the corresponding eigenvalues.
Since $\W_\indexopt \in \mathbb{R}^{\approxrank \times \approxrank}$ is a small matrix, its eigenvalue decomposition can be computed at a relatively low cost.

To enforce positive semidefiniteness, we define $\hat{\W}_\indexopt$ by truncating the negative eigenvalues of $\W_\indexopt$:
\begin{equation}
    \hat {\W}_\indexopt := \U_\indexopt \bm{\Sigma}_{\indexopt,+} \U_\indexopt^\top, \notag
\end{equation}
where
\begin{equation}
    \bm{\Sigma}_{n,+} = \mathrm{diag}(\lambda_1^+, \ldots, \lambda_\approxrank^+), \quad
    \lambda_i^+ = \begin{cases}
        \lambda_i\quad &(\lambda_i\ge 0), \\
        0 \quad &(\lambda_i\le0).
    \end{cases} \notag
\end{equation}
By construction, $\hat{\W}_\indexopt$ is positive semidefinite.

We then define the approximate Hessian by
\begin{equation}
    \approxH_\indexopt := \V_\indexopt \hat {\W}_\indexopt^{\dagger} \V_\indexopt^\top = \Z_\indexopt \Z_\indexopt^\top,
    \notag
\end{equation}
where
\begin{equation}
    \Z_\indexopt := \V_\indexopt \U_\indexopt \left(\bm{\Sigma}_{\indexopt,+}^{1/2}\right)^{\dagger}. \label{eq:mat_Z}
\end{equation}
Since $\approxH_\indexopt = \Z_\indexopt \Z_\indexopt^\top$, it follows immediately that $\approxH_\indexopt$ is positive semidefinite. Consequently, the operator $\operator_\indexopt = \tilde{\bm{H}}_\indexopt$ is also positive semidefinite, thereby preserving the modified dissipation law \eqref{eq:rsav_modified_dissipation_law}.

At each iteration, we need to solve a linear system of the form
$
\A_\indexopt \bm{x} = \nabla f(\param_\indexopt)
$.
A direct solution of this system requires $\mathcal{O}(d^3)$ operations, which becomes prohibitively expensive when $d$ is large.
To reduce this computational cost, we exploit the low-rank structure of the approximate Hessian and apply the Woodbury matrix identity. Specifically, we obtain
\begin{equation}
    \A_\indexopt^{-1}\nabla f(\param_\indexopt)
    = \nabla f(\param_\indexopt) - \Z_\indexopt \left(\frac{1}{\stepsize} \I^{(m)} + \Z_\indexopt^\top \Z_\indexopt \right)^{-1} \Z_\indexopt^\top \nabla f(\param_\indexopt).
    \label{eq:woodbury}
\end{equation}
The computational complexity of evaluating $\A_\indexopt^{-1} \nabla f(\param_\indexopt)$ via~\eqref{eq:woodbury} is $\mathcal{O}(m^2 d)$, which is substantially lower than the $\mathcal{O}(d^3)$ cost of solving the original linear system directly, since $m \ll d$.

We refer to the proposed method as the N-RSAV method. A summary of the N-RSAV method is provided in Algorithm~\ref{alg:proposed}.

\begin{algorithm}[tbp]
\caption{N-RSAV method}
\label{alg:proposed}
\begin{algorithmic}[1]
\Require Initial point $\param_0 \in \R^d$, step size $\Delta t > 0$, constant $C$, rank parameter $\approxrank$
\State $r_0 = \sqrt{f(\param_0)+C}$
\For{$n = 0,1,2,\dots,N-1$}
    \State Sample an index set $ \mathcal{M }\subset [d]$ with  $|\mathcal{M}| = m$ uniformly at random without replacement
    \State Compute $\V_\indexopt$ and $\W_\indexopt$ using Hessian-vector products 
    \State Compute $\hat{\W}_\indexopt$ by truncating the negative eigenvalues of $\W_\indexopt$
    \State Compute $\Z_\indexopt $ via~\eqref{eq:mat_Z}
    \State Compute $\A_\indexopt^{-1}\nabla f(\param_\indexopt)$ via~\eqref{eq:woodbury}
    \State Update $r_{\indexopt+1}$ and $\param_{\indexopt+1}$ by the RSAV scheme~\eqref{eq:rsav_scheme}
\EndFor
\State \Return $\param_N$
\end{algorithmic}
\end{algorithm}

\subsection{Adaptive Hessian reuse strategy based on  energy deviation}
\label{subsec:adaptive_proposed_method}
In the previous subsection~\ref{subsec:proposed_method}, we proposed a method 
that employs a randomized low-rank approximation for the Hessian.
Nevertheless, computing the approximate Hessian $\approxH_\indexopt$ at each iteration still incurs a non-negligible computational cost in the overall algorithm.
Therefore, it is desirable to reduce the frequency with which the approximate Hessian is constructed.

Our strategy is to reuse the same approximate Hessian as long as a certain criterion is satisfied.
To this end, we introduce an indicator that measures the discrepancy between  the shifted original energy $f(\param_\indexopt)+C$ and the modified energy $r_\indexopt^2$:
\begin{equation*}
    e_\indexopt = \left|\frac{r_\indexopt}{\sqrt{f(\param_\indexopt)+C}} - 1\right|.
\end{equation*}
By construction, $e_\indexopt$ is close to zero when the discrepancy is small and increases as the discrepancy becomes larger.
We use this indicator to determine whether the approximate Hessian should be reused or recomputed.
Specifically, if the current deviation $e_\indexopt$
 satisfies $e_\indexopt < \recomputefactor e_{\indexopt-1}$
 for a prescribed tolerance factor $\recomputefactor >1$, 
 the deviation is regarded as being under control, and we reuse the previous approximation: $\approxH_\indexopt = \approxH_{\indexopt-1}$.
 Otherwise, the increase in the deviation is regarded as significant, and a new approximate Hessian $\approxH_\indexopt$ is constructed.
 
 We refer to the resulting adaptive method as AN-RSAV. A summary of the AN-RSAV algorithm is provided in Algorithm~\ref{alg:proposed_adaptive_ver}.

\begin{algorithm}[tbp]
\caption{AN-RSAV method}
\label{alg:proposed_adaptive_ver}
\begin{algorithmic}[1]
\Require Initial point $\param_0 \in \R^d$, step size $\stepsize > 0$, constant $C$, rank parameter $\approxrank$,   tolerance factor $\recomputefactor >1$
\State $r_0 = \sqrt{f(\param_0)+C}, e_{-1}=0$
\For{$n = 0,1,2,\dots,N-1$}
\If{$e_\indexopt < \recomputefactor e_{\indexopt-1}$ }
    \State $\Z_\indexopt = \Z_{\indexopt-1}$
    \Else
    \State Sample an index set $ \mathcal{M }\subset [d]$ with  $|\mathcal{M}| = m$ uniformly at random without replacement
    \State Compute $\V_\indexopt$ and $\W_\indexopt$ using Hessian-vector products 
    \State Compute $\hat{\W}_\indexopt$ by truncating the negative eigenvalues of $\W_\indexopt$
    \State Compute $\Z_\indexopt $ via~\eqref{eq:mat_Z}
    \EndIf
    \State Compute $\A_\indexopt^{-1}\nabla f(\param_\indexopt)$ via~\eqref{eq:woodbury}
    \State Update $r_{\indexopt+1}$ and $\param_{\indexopt+1}$ by the RSAV scheme \eqref{eq:rsav_scheme}
\EndFor
\State \Return $\param_N$
\end{algorithmic}
\end{algorithm}

\subsection{Remarks}
\label{subsec:remarks}
In this subsection, we discuss the expected advantages and potential limitations of the proposed method. 

The expected advantages are twofold.
First, near a local minimizer, where the Hessian provides informative curvature information, the proposed scheme performs a curvature-aware scaling of the gradient, which is expected to accelerate convergence.
Second, since the proposed method employs a low-rank approximation of the Hessian, it may still provide effective curvature information even when the approximation rank $m$ is small, provided that the Hessian has an effectively low-rank structure.

On the other hand, to preserve positive semidefiniteness, the proposed method truncates the negative eigenvalues of the approximate Hessian. Consequently, information associated with directions of negative curvature is discarded.
As a result, when the objective function is nonconvex, the method may exhibit stagnation near saddle points~\cite{saddle_point}.

To address this issue, it can be advantageous in practice to combine the proposed method with first-order stochastic optimization methods such as Adam~\cite{adam}. 
A natural strategy is to use Adam during the initial stage of the optimization process to escape saddle points, and then switch to the proposed method once the iterates enter a region in which the objective function exhibits locally convex behavior around a local minimizer.

\section{Convergence analysis}
\label{sec:convergence_analysis}

In this section, we first analyze the convergence of the RSAV scheme with a general positive semidefinite operator $\operator_\indexopt$, and then analyze the proposed N-RSAV scheme, in which the Nystr\"om approximate Hessian $\tilde{\bm{H}}_\indexopt$ is used as $\operator_\indexopt$.
To the best of our knowledge, this is the first convergence analysis of an RSAV scheme with a general operator $\operator_\indexopt$.
Existing analyses have been restricted to specific choices of $\operator_\indexopt$; for example,
the previous study \cite{sav_opt2} considered only the case $\operator_\indexopt=\bm{0}$.

Throughout the convergence analysis, we use the  notation:
\begin{equation*}
    f_\ast \coloneqq \min_{\param\in\mathbb{R}^d} f(\param),
    \qquad
    \delta \coloneqq f_\ast + C >0.
\end{equation*}
We also recall the assumptions imposed on the objective function.

\begin{definition}[$L$-smoothness]
\label{def:L_smooth}
A function $f:\mathbb{R}^d\to\mathbb{R}$ is said to be $L$-smooth with $L>0$ if its gradient $\nabla f$ is Lipschitz continuous, that is, if
\[
    \|\nabla f(\param)-\nabla f(\bm{\phi})\| \le L\|\param-\bm{\phi}\|
\]
holds for all $\param,\bm{\phi}\in\mathbb{R}^d$.
\end{definition}

\begin{remark}
The following inequality is 
an equivalent characterization of $L$-smoothness:
\begin{equation}
    f(\bm{\phi})
    \le f(\param)+\nabla f(\param)^\top(\bm{\phi}-\param)
    +\frac{L}{2}\|\bm{\phi}-\param\|^2   \quad \text{for all} \quad \param, \bm{\phi} \in \R^d . \notag
\end{equation}
\end{remark}

\begin{definition}[Polyak--\L{}ojasiewicz (PL) condition]
A function $f:\mathbb{R}^d\to\mathbb{R}$ is said to satisfy the PL condition if there exists a constant $\mu>0$ such that
\begin{equation}
    \frac{1}{2}\|\nabla f(\param)\|^2
    \ge \mu\bigl(f(\param)-f_\ast\bigr), \notag
\end{equation}
holds for all $\param\in\mathbb{R}^d$.
\end{definition}

The PL condition generalizes strong convexity and is widely used in the convergence analysis of nonconvex optimization problems.
It is also known that, under suitable assumptions, the loss function of overparameterized neural networks satisfies this condition.

\subsection{\texorpdfstring{Results for RSAV with a general positive semidefinite operator $\operator_\indexopt$}{operator Ln}}

The following lemma will be used to prove the monotone decrease of the objective function and convergence to a minimizer.

\begin{lemma}
\label{thm:positive_lower_bound}
Assume that $f$ is $L$-smooth. Then there exists a constant $C_1>0$ such that, if the step size satisfies $\stepsize\le C_1$, the scalar auxiliary variable $r_\indexopt$ generated by the RSAV scheme converges and has a positive lower bound:
\begin{equation}
    \lim_{\indexopt\to\infty} r_\indexopt
    = r_\ast \ge \frac{\sqrt{\delta}}{2}>0 .
    \label{r_positive_lower_bound}
\end{equation}
\end{lemma}

\begin{proof}
Let
\begin{equation}
        g(\param)\coloneqq \sqrt{f(\param)+C}. \notag
\end{equation}
Then the RSAV scheme can be written as
\begin{equation}
 \left\{ \ 
 \begin{aligned}
    &\frac{\tilr_{\indexopt+1}- r_\indexopt}{\stepsize} = \nabla g(\param_\indexopt)^\top \frac{\param_{\indexopt+1}-\param_\indexopt}{\stepsize},\\
    &\frac{\param_{\indexopt+1}-\param_\indexopt }{\stepsize}= -\operator_\indexopt \param_{\indexopt+1}- \frac{\tilr_{\indexopt+1}}{\sqrt{f(\param_\indexopt)+C}}\nabla f(\param_\indexopt)+\operator_\indexopt \param_\indexopt\\
    &r_{\indexopt+1} = \xi_\indexopt \tilde{r}_{\indexopt+1} + (1-\xi_\indexopt) g(\param_{\indexopt+1}).
\end{aligned}
 \right.   
\label{eq:rsav_discrete_scheme}
\end{equation}

The proof proceeds in four steps. 
In Step 1, we show that the sequence $\{r_\indexopt\}_{\indexopt}$ is nonnegative and monotonically decreasing, and therefore convergent. 
In Step 2, we verify that $g$ is $L_g$-smooth with $L_g\coloneqq L/\sqrt{\delta}$. 
In Step 3, we derive an estimate for the discrepancy between $r_\indexopt$ and $g(\param_\indexopt)$. 
Finally, in Step 4, we prove that, for a sufficiently small step size $\stepsize$, 
the limit of $\{r_\indexopt\}_{\indexopt}$ is bounded away from zero.

\paragraph{Step~1}

By the modified dissipation law, the sequence $\{r_\indexopt^2\}_\indexopt$ is monotonically decreasing and bounded below. Therefore, it converges, and there exists 
$r_\ast\geq 0$ 
such that
\begin{equation}
    \lim_{\indexopt\to\infty} r_\indexopt^2 = r_\ast^2 .
    \notag
\end{equation}
Since $\A_\indexopt$ is positive definite, the coefficient appearing in the first equation
of the RSAV scheme~\eqref{eq:rsav_scheme} satisfies
\begin{equation}
        \left(1+\stepsize
    \frac{\nabla f(\param_\indexopt)^\top\bm{A}_\indexopt^{-1}\nabla f(\param_\indexopt)}{2g(\param_\indexopt)^2}
    \right)^{-1}\ge0 . \notag
\end{equation}
Together with $r_0>0$ and the convex combination 
\begin{equation}
        r_{\indexopt+1}
    =\xi_\indexopt\tilr_{\indexopt+1}
    +(1-\xi_\indexopt)g(\param_{\indexopt+1}) \notag
\end{equation}
in the RSAV scheme~\eqref{eq:rsav_scheme}, this implies by induction that
\begin{equation}
        r_\indexopt\ge0,
    \qquad
    \tilr_{\indexopt+1}\ge0
    \quad \text{for all} \quad \indexopt \ge  0. \notag
\end{equation}
Therefore, the sequence $\{r_\indexopt\}_\indexopt$ is also decreasing, and it follows that
\begin{equation}
    \lim_{\indexopt\to\infty} r_\indexopt = r_\ast .
    \notag
\end{equation}

\paragraph{Step~2}
By the $L$-smoothness of $f$, we have
\begin{equation}
    f\left(\param-\frac{1}{L}\nabla f(\param)\right)
    \le
    f(\param)-\frac{1}{L}\|\nabla f(\param)\|^2
    +\frac{1}{2L}\|\nabla f(\param)\|^2 .
    \notag
\end{equation}
Since $f\left(\param-\frac{1}{L}\nabla f(\param)\right)+C>0$, we obtain
\begin{equation}
    \|\nabla f(\param)\|^2
    \le
    2L\left\{f(\param)-f\left(\param-\frac{1}{L}\nabla f(\param)\right)\right\}
    < 2L\{f(\param)+C\}
    =2L\{g(\param)\}^2 .
    \label{grad_f_bound}
\end{equation}
Therefore, it follows that
\begin{equation}
    \|\nabla g(\param)\|
    =
    \frac{\| \nabla f(\param)\|}{2g(\param)} \le \frac{\sqrt{2L}g(\param)}{2 g(\param)} = \sqrt{\frac{L}{2}}
    \notag.
\end{equation}
This implies that $g$ is Lipschitz continuous.
For arbitrary $\param,\bm{\phi}\in\mathbb{R}^d$, we have
\begin{equation}
    \begin{aligned}
        \|\nabla g(\param) - \nabla g(\bm{\phi})\|
        &\le \left\|\frac{\nabla f(\param)- \nabla f(\bm{\phi})}{2 g(\param)} + 
        \frac{\nabla f(\bm{\phi})}{2}\left\{\frac{1}{g(\param)} - \frac{1}{g(\bm{\phi})}\right\} \right\| \\
        &\le \frac{\|\nabla f(\param)- \nabla f(\bm{\phi})\|}{2 g(\param)}+ \frac{\|\nabla f(\bm{\phi})\|}{2} \frac{|g(\bm{\phi})-g(\param)|}{g(\param) g(\bm{\phi})}\\
        &\le \left\{ \frac{L}{2 \sqrt{\delta}} + \frac{\sqrt{2L}g(\bm{\phi})}{2} \frac{\sqrt{L/2}}{\sqrt{\delta}g(\bm{\phi})} \right\}\|\param- \bm{\phi}\|
        \le \frac{L}{\sqrt{\delta}}\|\param- \bm{\phi}\| . \notag
    \end{aligned}
\end{equation}
Hence, $g$ is $L_g$-smooth with $L_g=L/\sqrt{\delta}$.

\paragraph{Step~3}

For convenience, we introduce the notation
\begin{align*}
    & \s_\indexopt \coloneqq \param_\indexopt-\param_{\indexopt-1}, \\
    & h_n
    \coloneqq
    g(\param_\indexopt)-g(\param_{\indexopt-1})
    -
    \nabla g(\param_{\indexopt-1})^\top \s_\indexopt, \\
    & \nu_k\coloneqq \xi_{\indexopt-2}\xi_{\indexopt-3}\cdots\xi_k\in[0,1]
    \quad (k=0,\ldots,\indexopt-2), \quad \nu_{\indexopt-1}=1.
\end{align*}

From the convex combination in the RSAV scheme~\eqref{eq:rsav_scheme}, 
\begin{equation}
    r_\indexopt-g(\param_\indexopt)
    =\xi_{\indexopt-1}\bigl(\tilr_\indexopt-g(\param_\indexopt)\bigr) \quad \text{for all} \quad \indexopt\geq 1.
    \label{conv_r}
\end{equation}
Moreover, from the first equation of
the RSAV scheme~\eqref{eq:rsav_discrete_scheme}, we have
\begin{equation}
    \tilr_\indexopt -g(\param_\indexopt) = r_{\indexopt-1} - g(\param_{\indexopt-1}) - h_\indexopt \quad \text{for all} \quad \indexopt\geq 1.
\label{eq:tilr-g}
\end{equation}
Repeatedly applying~\eqref{conv_r} and~\eqref{eq:tilr-g} yields
\begin{equation}
\begin{aligned}
    \tilr_\indexopt-g(\param_\indexopt)
    &=r_{\indexopt-1}-g(\param_{\indexopt-1})
    -h_\indexopt \\
    &=\xi_{\indexopt-2}\bigl(\tilr_{\indexopt-1}-g(\param_{\indexopt-1})\bigr)
    - h_\indexopt \\
    &=\xi_{\indexopt-2}\left(
    r_{\indexopt-2}-g(\param_{\indexopt-2})
    -h_{\indexopt-1}
    \right) -h_\indexopt\\
    &=\cdots \\
    &=\nu_0(r_0-g(\param_0))
    -\sum_{k=1}^{\indexopt}\nu_{k-1}h_k \\
    &=-\sum_{k=1}^{\indexopt}\nu_{k-1}h_k, 
\end{aligned}
\label{eq:tilde_r-g_val}
\end{equation}
where the last equality follows from the identity $r_0=g(\param_0)$.
Since $g$ is $L_g$-smooth, we have
\begin{equation}
    \begin{aligned}
    |h_k| &= |g(\param_k)- g(\param_{k-1})- \nabla g(\param_{k-1})^\top \s_k| \\
    &= \left|\int_0^1 \nabla g(\param_{k-1}+t \s_k)^\top \s_k  \, dt - \nabla g(\param_{k-1})^\top \s_k\right| \\
    &\le \int_0^1 \left| \left[ \nabla g(\param_{k-1}+t \s_k) - \nabla g(\param_{k-1})\right]^\top \s_k \right| \, dt \\
    &=\int_0^1 \|  \nabla g(\param_{k-1}+t \s_k) - \nabla g(\param_{k-1})\|  \|\s_k\| \, dt \\
    & \le  \int_0^1 L_g \|t \s_k\| \|\s_k\|  \, dt = \frac{L_g}{2}\|\s_k\|^2
         \label{eq:h_val}.
    \end{aligned}
\end{equation}
From the modified dissipation law~\eqref{eq:rsav_modified_dissipation_law}, we see that
\begin{equation}
    G_{k-1} \le \frac{1}{1-\eta}(r_{k-1} ^2-r_{k}^2). \notag
\end{equation}
By the definition of $G_{k-1}$ in~\eqref{eq:cond}, we obtain
\begin{equation}
\begin{aligned}
    \|\s_k \|^2=\|\param_{k}-\param_{k-1}\|^2
    &\le
    \frac{\stepsize}{1-\eta}(r_{k-1}^2-r_{k}^2)
    -\stepsize(\param_{k}-\param_{k-1})^\top\operator_{k-1}(\param_{k}-\param_{k-1}) \\
    &\le
    \frac{\stepsize}{1-\eta}(r_{k-1}^2-r_{k}^2),
\end{aligned}
\label{eq:param_diff}
\end{equation}
where the last inequality follows from the positive semidefiniteness of $\operator_{k-1}$.

Combining~\eqref{eq:tilde_r-g_val} with the estimates~\eqref{eq:h_val} and~\eqref{eq:param_diff}, we obtain
\begin{equation}
\begin{aligned}
    |\tilr_\indexopt-g(\param_\indexopt)|
    &\le
    \sum_{k=1}^{\indexopt}\nu_{k-1}
    |h_k|\\
    &\le
    \frac{L_g}{2}\sum_{k=1}^{\indexopt}\nu_{k-1}\|\s_k\|^2 \le 
    \frac{L_g}{2}\sum_{k=1}^{\indexopt}\|\s_k\|^2  \\
    & \le
    \frac{L_g\stepsize}{2(1-\eta)}
    \sum_{k=1}^{\indexopt}(r_{k-1}^2-r_k^2) 
    =
    \frac{L_g\stepsize}{2(1-\eta)}(r_0^2-r_\indexopt^2) \quad \text{for all} \quad \indexopt \geq 1.
\end{aligned}
\label{eq:tilr_g_eval}
\end{equation}
Consequently, by~\eqref{conv_r} and~\eqref{eq:tilr_g_eval}, it follows that
\begin{equation}
    |r_\indexopt-g(\param_\indexopt)|
    =\xi_{\indexopt-1}|\tilr_\indexopt-g(\param_\indexopt)| \le
    \frac{L_g\stepsize}{2(1-\eta)}(r_0^2-r_\indexopt^2) \quad \text{for all} \quad \indexopt \geq 1. \label{eq:g-r_eval}
\end{equation}

\paragraph{Step~4}
If, for every $N_0\in\mathbb{N}$, there exists $\indexopt\ge N_0$ such that $r_\indexopt >\sqrt{\delta}$, then
\[
    r_\ast\ge\sqrt{\delta},
\]
and hence the desired bound~\eqref{r_positive_lower_bound} follows.
It therefore remains to consider the case in which there exists $N_0\in\mathbb{N}$ such that
\[
    r_n\le\sqrt{\delta}
    \quad \text{for all} \quad \indexopt \ge N_0.
\]
Note that since $g(\param_\indexopt)\ge\sqrt{\delta}$, we then have
\begin{equation}
    r_\indexopt \le g(\param_\indexopt)
    \quad \text{for all} \quad \indexopt\ge N_0.
\end{equation}

We first consider the case where, for every
$N_1\geq N_0$,
there exists $\indexopt \ge N_1$ such that $\xi_{\indexopt-1}=0$. For such an index $\indexopt
$, the third equation of the RSAV scheme~\eqref{eq:rsav_scheme} reduces to
\begin{equation}
        r_\indexopt=g(\param_\indexopt). \notag
\end{equation}
On the other hand, $r_\indexopt\le\sqrt{\delta}$ for $\indexopt \ge N_0$, whereas $g(\param_\indexopt)\ge\sqrt{\delta}$ always holds. 
Hence, for such an index $\indexopt$,
we have
\begin{equation}
        r_\indexopt=g(\param_\indexopt)=\sqrt{\delta}. \notag
\end{equation}
Since such indices can be chosen arbitrarily large and the sequence $\{r_\indexopt\}_\indexopt$ is monotonically decreasing, we obtain
\begin{equation}
        r_\ast=\sqrt{\delta}\ge\frac{\sqrt{\delta}}{2}, \notag
\end{equation}
and the conclusion follows.

Therefore, we may assume that there exists $N_2\in\mathbb{N}$ such that
\begin{equation}
        \xi_{\indexopt-1}>0
    \quad \text{for all} \quad \indexopt\ge N_2. \notag
\end{equation}
In what follows, we set $N\coloneqq\max\{N_0,N_2\}$.

Since the RSAV scheme~\eqref{eq:rsav_scheme} expresses $r_\indexopt$ as a convex combination of $\tilr_\indexopt$ and $g(\param_\indexopt)$, 
it follows from $r_\indexopt\le g(\param_\indexopt)$ and $\xi_{\indexopt-1}>0$ that
\begin{equation}
    0\le \tilr_\indexopt\le r_\indexopt\le g(\param_\indexopt)
    \quad \text{for all} \quad \indexopt\ge N.
    \label{r_ineq}
\end{equation}
By the $L_g$-smoothness of $g$, the RSAV scheme~\eqref{eq:rsav_discrete_scheme},  \eqref{eq:param_diff} and~\eqref{r_ineq}, we obtain
\begin{equation}
\begin{aligned}
    g(\param_{\indexopt+1})
    &\le
    g(\param_\indexopt)
    +\nabla g(\param_\indexopt)^\top(\param_{\indexopt+1}-\param_\indexopt)
    +\frac{L_g}{2}\|\param_{\indexopt+1}-\param_\indexopt\|^2 \\
    &=
    g(\param_\indexopt)+(\tilr_{\indexopt+1}-r_\indexopt)
    +\frac{L_g}{2}\|\param_{\indexopt+1}-\param_\indexopt\|^2 \\
    &\le
    g(\param_\indexopt)+(r_{\indexopt+1}-r_\indexopt)
    +\frac{\stepsize L_g}{2(1-\eta)}(r_\indexopt^2-r_{\indexopt+1}^2)
\end{aligned}
\label{eq:g_smooth}
\end{equation}
for all $\indexopt\ge N$.
Summing~\eqref{eq:g_smooth} from $n=N$ to $K$ yields
\begin{equation}
    g(\param_{K+1})
    \le
    g(\param_N)+(r_{K+1}-r_N)
    +\frac{\stepsize L_g}{2(1-\eta)}(r_N^2-r_{K+1}^2).
    \notag
\end{equation}
Since $g(\param_{K+1})\ge\sqrt{\delta}$, we get
\begin{equation}
    \sqrt{\delta}
    \le
    g(\param_N)+(r_{K+1}-r_N)
    +\frac{\stepsize L_g}{2(1-\eta)}(r_N^2-r_{K+1}^2).
    \notag
\end{equation}
Letting $K\to\infty$ gives
\begin{equation}
    \sqrt{\delta}
    \le
    g(\param_N)+(r_\ast-r_N)
    +\frac{\stepsize L_g}{2(1-\eta)}(r_N^2-r_\ast^2). \notag
\end{equation}
Rearranging yields
\begin{equation}
\begin{aligned}
    r_\ast
    &\ge
    \sqrt{\delta}+(r_N-g(\param_N))
    -\frac{\stepsize L_g}{2(1-\eta)}r_N^2
    +\frac{\stepsize L_g}{2(1-\eta)}r_\ast^2 \\
    &\ge
    \sqrt{\delta}+(r_N-g(\param_N))
    -\frac{\stepsize L_g}{2(1-\eta)}r_N^2 .
\end{aligned}
\label{eq:lower_bound_pre}
\end{equation}

Define
\begin{equation}
    C_1\coloneqq \frac{\sqrt{\delta}(1-\eta)}{L_g r_0^2}>0 .
    \notag
\end{equation}
If $\stepsize\le C_1$, then~\eqref{eq:lower_bound_pre}, together with~\eqref{eq:g-r_eval},
implies that
\begin{equation}
\begin{aligned}
    r_\ast
    &\ge
    \sqrt{\delta}+(r_N-g(\param_N))
    -\frac{\stepsize L_g}{2(1-\eta)}r_N^2 \\
    &\ge
    \sqrt{\delta}
    -\frac{\stepsize L_g}{2(1-\eta)}(r_0^2-r_N^2)
    -\frac{\stepsize L_g}{2(1-\eta)}r_N^2 \\
    &=
    \sqrt{\delta}
    -\stepsize\frac{L_g r_0^2}{2(1-\eta)}
    \ge \frac{\sqrt{\delta}}{2}.
\end{aligned}
\notag
\end{equation}
This proves~\eqref{r_positive_lower_bound}.
\end{proof}

The following theorem summarizes the convergence properties of the RSAV scheme.

\begin{theorem}
\label{thm:rsav_convergence_analysis}
Assume that $f$ is $L$-smooth. In addition, suppose that the positive semidefinite linear operator $\operator_\indexopt$ is uniformly bounded above in the Loewner order; namely, there exists a constant $M>0$ such that
\begin{equation}
    \bm{0}\preceq \operator_\indexopt \preceq M\I^{(d)} .
    \label{operator_assumption}
\end{equation}
Let $C_1>0$ be the constant defined in Lemma~\ref{thm:positive_lower_bound}. Suppose that the step size $\stepsize$ satisfies
\begin{equation}
    \stepsize
    \le
    \min\left\{C_1,\frac{\sqrt{\delta}}{L\sqrt{f(\param_0)+C}}\right\} .
    \label{stepsize_assumption}
\end{equation}
Let 
\begin{equation*}
    \{(\param_\indexopt,r_\indexopt,\tilr_{\indexopt+1})\}_{\indexopt}
\end{equation*}
be the sequence generated by the RSAV scheme.
Then the following statements hold.

\begin{enumerate}
\item[(i)] The objective function values are monotonically decreasing:
\begin{equation}
    f(\param_{\indexopt+1}) \le f(\param_\indexopt) .
    \label{f_dcreasing}
\end{equation}
Moreover, 
\begin{equation}
    \lim_{\indexopt\to\infty}\|\nabla f(\param_\indexopt)\|=0 . \notag
\end{equation}

\item[(ii)] If $f$ satisfies the PL condition, then the sequence $\{\param_\indexopt\}_\indexopt$ converges to a global minimizer $\param_\infty$:
\begin{equation}
    \lim_{\indexopt\to\infty}\param_\indexopt=\param_\infty,
    \qquad
    f(\param_\infty)=f_\ast . \notag
\end{equation}

\item[(iii)] If $f$ satisfies the PL condition, then 
the function values converge linearly in the sense that
\begin{equation}
    f(\param_{\indexopt+1})-f_\ast
    \le q\bigl(f(\param_\indexopt)-f_\ast\bigr), \notag
\end{equation}
where
\begin{equation}
    0<q\coloneqq
    1-\frac{\mu\tilr_\ast\stepsize}{(1+\stepsize M)\sqrt{f(\param_0)+C}}
    <1 . \notag
\end{equation}
\end{enumerate}
\end{theorem}

\begin{proof}
In the proof, we define 
\begin{equation}
    \alpha_\indexopt
    \coloneqq
   \stepsize \frac{\tilr_{\indexopt+1}}{\sqrt{f(\param_\indexopt)+C}}, \notag
\end{equation}
which can be interpreted as an effective step size for the gradient flow $\dot{\param} = - \A_\indexopt^{-1} \nabla f (\param)$. 
With this notation, the RSAV scheme \eqref{eq:rsav_scheme} can be written as
\begin{equation}
\left\{
\begin{aligned}
    &\tilde{r}_{\indexopt+1}
    =
    \left(1+
    \stepsize\frac{\|\nabla f(\param_\indexopt)\|_{\A_\indexopt^{-1}}^2}{2(f(\param_\indexopt)+C)}
    \right)^{-1}r_\indexopt, \\
    &\param_{\indexopt+1}
    =\param_\indexopt-
    \alpha_\indexopt\A_\indexopt^{-1}\nabla f(\param_\indexopt), \\
    &r_{\indexopt+1}
    =\xi_\indexopt\tilde{r}_{\indexopt+1}
    +(1-\xi_\indexopt)\sqrt{f(\param_{\indexopt+1})+C}.
\end{aligned}
\right.
\label{relaxed_sav_scheme_for_proof}
\end{equation}
Since the assumption \eqref{stepsize_assumption} implies $\stepsize\le C_1$, Lemma~\ref{thm:positive_lower_bound} ensures that $r_\indexopt>0$ for all $\indexopt$. 
Therefore, the first equation of \eqref{relaxed_sav_scheme_for_proof} implies that $\tilr_{\indexopt+1}>0$, and hence $\alpha_\indexopt>0$.

\paragraph{Proof of (i)}
By the $L$-smoothness of $f$, we have
\begin{equation}
    f(\param_{\indexopt+1})
    \le
    f(\param_\indexopt)
    +\nabla f(\param_\indexopt)^\top(\param_{\indexopt+1}-\param_\indexopt)
    +\frac{L}{2}\|\param_{\indexopt+1}-\param_\indexopt\|^2 .
    \notag
\end{equation}
Using the second equation of \eqref{relaxed_sav_scheme_for_proof}, we obtain
\begin{equation}
    f(\param_{\indexopt+1})
    \le
    f(\param_\indexopt)
    -\alpha_\indexopt\nabla f(\param_\indexopt)^\top\A_\indexopt^{-1}\nabla f(\param_\indexopt)
    +\frac{L}{2}\alpha_\indexopt^2
    \|\A_\indexopt^{-1}\nabla f(\param_\indexopt)\|^2 .
    \label{f_L_smooth_theorem_3_i_pre}
\end{equation}
Since $\A_\indexopt=\I^{(d)}+\stepsize\operator_\indexopt$, the largest eigenvalue of $\A_\indexopt^{-1}$ is at most one. 
Hence,
\begin{equation}
    \|\A_\indexopt^{-1}\nabla f(\param_\indexopt)\|^2
    =
    \nabla f(\param_\indexopt)^\top(\A_\indexopt^{-1})^2\nabla f(\param_\indexopt) 
    \le
    \nabla f(\param_\indexopt)^\top\A_\indexopt^{-1}\nabla f(\param_\indexopt).
\notag
\end{equation}
Substituting this estimate into \eqref{f_L_smooth_theorem_3_i_pre} gives
\begin{equation}
    f(\param_{\indexopt+1})
    \le
    f(\param_\indexopt)
    -\alpha_\indexopt
    \left(1-\frac{\alpha_\indexopt L}{2}\right)
    \|\nabla f(\param_\indexopt)\|_{\A_\indexopt^{-1}}^2 .
    \label{Lsmooth_eq}
\end{equation}
By $\tilr_{\indexopt+1}\le r_\indexopt\le r_0$,   $\sqrt{f(\param_0)+C}\ge\sqrt{\delta}$ and \eqref{stepsize_assumption}, 
we have
\begin{equation}
    1-\frac{\alpha_\indexopt L}{2}
    =
    1-\frac{\tilr_{\indexopt+1}L}{2\sqrt{f(\param_\indexopt)+C}}\stepsize 
    \ge
    1-\frac{r_0L}{2\sqrt{\delta}}\stepsize 
    \ge \frac{1}{2}>0.
\label{eq:alpha_f_monotonic}
\end{equation}
From \eqref{Lsmooth_eq} and \eqref{eq:alpha_f_monotonic}, we conclude that $f(\param_{\indexopt+1})
    \le
    f(\param_\indexopt)$.

We next show that $\|\nabla f(\param_\indexopt)\|\to0$. 
To this end, we first establish a positive lower bound for $\tilr_{\indexopt+1}$.

By the $L$-smoothness of $f$,
\begin{equation}
    f(\bm{\phi})
    \le
    f(\param)+\nabla f(\param)^\top(\bm{\phi}-\param)
    +\frac{L}{2}\|\bm{\phi}-\param\|^2
    \quad \text{for all } \param,\bm{\phi}\in\R^d.
    \notag
\end{equation}
Setting $\bm{\phi}=\param_\indexopt-(1/L)\nabla f(\param_\indexopt)$ and $\param=\param_\indexopt$ yields
\begin{equation}
    \|\nabla f(\param_\indexopt)\|^2
    \le
    2L\left\{
    f(\param_\indexopt)
    -f\left(\param_\indexopt-\frac{1}{L}\nabla f(\param_\indexopt)\right)
    \right\}.
    \label{eq:nabla_f_le_2L}
\end{equation}
Let
\[
    U\coloneqq 2L\bigl(f(\param_0)-f_\ast\bigr)\ge0.
\]
Since $f(\param_\indexopt)\le f(\param_0)$ by the monotonicity shown above, the inequality \eqref{eq:nabla_f_le_2L} implies that
\begin{equation}
    \|\nabla f(\param_\indexopt)\|^2\le U .
    \label{grad_norm_upper_bound}
\end{equation}
Moreover, by the assumption \eqref{operator_assumption}, it follows that
\begin{equation}
    \frac{1}{1+M\stepsize}\|\nabla f(\param_\indexopt)\|^2
    \le
    \|\nabla f(\param_\indexopt)\|_{\A_\indexopt^{-1}}^2
    \le
    \|\nabla f(\param_\indexopt)\|^2 .
    \label{grad_val_from_assumption}
\end{equation}
Since Lemma~\ref{thm:positive_lower_bound} applies under \eqref{stepsize_assumption}, we have $r_\indexopt\ge r_\ast>0$. 
Combining this with \eqref{grad_norm_upper_bound} and \eqref{grad_val_from_assumption}, the first equation of \eqref{relaxed_sav_scheme_for_proof} yields
\begin{align*}
    \tilr_{\indexopt+1}
    & =
    \left(1+\stepsize
    \frac{\|\nabla f(\param_\indexopt)\|_{\A_\indexopt^{-1}}^2}{2(f(\param_\indexopt)+C)}
    \right)^{-1}r_\indexopt \\
    & \ge
    \left(1+\stepsize
    \frac{\|\nabla f(\param_\indexopt)\|^2}{2\delta}
    \right)^{-1}r_\ast 
    \ge
    \left(1+\stepsize\frac{U}{2\delta}\right)^{-1}r_\ast .
\end{align*}
Therefore, defining
\[
    \tilr_\ast\coloneqq
    \left(1+\stepsize\frac{U}{2\delta}\right)^{-1}r_\ast>0,
\]
we see that
\begin{equation}
    \tilr_{\indexopt+1}\ge \tilr_\ast>0
    \quad \text{for all} \quad \indexopt \geq 0.
    \label{tilr_positive_lower_bound}
\end{equation}

From \eqref{eq:alpha_f_monotonic} and \eqref{Lsmooth_eq}, we obtain
\begin{equation}
    \frac{\alpha_\indexopt}{2}
    \|\nabla f(\param_\indexopt)\|_{\A_\indexopt^{-1}}^2 \leq
    \left(1-\frac{\alpha_\indexopt L}{2}\right)
    \|\nabla f(\param_\indexopt)\|_{\A_\indexopt^{-1}}^2
    \le
    f(\param_\indexopt)-f(\param_{\indexopt+1}).
    \label{eq:L_smooth_eq_for_grad_zero}
\end{equation}
The monotonicity of $f$ and \eqref{tilr_positive_lower_bound} imply that
\begin{equation}
    \alpha_\indexopt
    =
    \frac{\tilr_{\indexopt+1}}{\sqrt{f(\param_\indexopt)+C}}\stepsize
    \ge
    \frac{\tilr_\ast}{\sqrt{f(\param_0)+C}}\stepsize .
    \label{alpha_val}
\end{equation}
Using \eqref{eq:L_smooth_eq_for_grad_zero}, \eqref{alpha_val} and \eqref{grad_val_from_assumption}, we obtain
\begin{align*}
    f(\param_\indexopt)-f(\param_{\indexopt+1})
    &\ge \frac{\alpha_\indexopt}{2}\|\nabla f(\param_\indexopt)\|^2_{\A_\indexopt^{-1}} \\
    &\ge \frac{\tilde{r}_\ast \stepsize}{2 \sqrt{f(\param_0)+C}} \|\nabla f(\param_\indexopt)\|^2_{\A_\indexopt^{-1}}\\
   & \ge \frac{\tilr_\ast\stepsize}{2(1+M\stepsize)\sqrt{f(\param_0)+C}}
    \|\nabla f(\param_\indexopt)\|^2.
\end{align*}
Summing this inequality from $n=0$ to $N$ gives
\begin{equation}
\begin{aligned}
    \sum_{n=0}^{N}\|\nabla f(\param_\indexopt)\|^2
    &\le
    \frac{2(1+M\stepsize)\sqrt{f(\param_0)+C}}{\tilr_\ast\stepsize}
    \bigl(f(\param_0)-f(\param_{N+1})\bigr) \\
    &\le
    \frac{2(1+M\stepsize)\sqrt{f(\param_0)+C}}{\tilr_\ast\stepsize}
    \bigl(f(\param_0)-f_\ast\bigr).
\end{aligned}
\notag
\end{equation}
Letting $N\to\infty$ yields
\begin{equation}
    \sum_{n=0}^{\infty}\|\nabla f(\param_\indexopt)\|^2<\infty .
    \notag
\end{equation}
Therefore, $\|\nabla f(\param_\indexopt)\|\to0$.

\paragraph{Proof of (ii)}
By \eqref{eq:alpha_f_monotonic}, we have $\alpha_\indexopt\le1/L$. Hence, by the $L$-smoothness of $f$,
\begin{equation}
\begin{aligned}
    f(\param_{\indexopt+1})-f(\param_\indexopt)
    &\le
    \nabla f(\param_\indexopt)^\top(\param_{\indexopt+1}-\param_\indexopt)
    +\frac{L}{2}\|\param_{\indexopt+1}-\param_\indexopt\|^2 \\
    &=-\frac{1}{\alpha_\indexopt}\|\param_{\indexopt+1}-\param_\indexopt\|_{\A_\indexopt}^2
    +\frac{L}{2}\|\param_{\indexopt+1}-\param_\indexopt\|^2 \\
    &\le
    -\frac{1}{\alpha_\indexopt}\|\param_{\indexopt+1}-\param_\indexopt\|^2
    +\frac{L}{2}\|\param_{\indexopt+1}-\param_\indexopt\|^2 \\
    &\le
    -\frac{1}{\alpha_\indexopt}\|\param_{\indexopt+1}-\param_\indexopt\|^2
    +\frac{1}{2\alpha_\indexopt}\|\param_{\indexopt+1}-\param_\indexopt\|^2 \\
    &=-\frac{1}{2\alpha_\indexopt}\|\param_{\indexopt+1}-\param_\indexopt\|^2 .
\end{aligned}
\notag
\end{equation}
Define
\[
    w_\indexopt\coloneqq f(\param_\indexopt)-f_\ast\ge0.
\]
Then the above inequality can be written as
\begin{equation}
    w_\indexopt-w_{\indexopt+1}
    \ge
    \frac{1}{2\alpha_\indexopt}\|\paramdiff\|^2 .
    \label{eq:L_smooth_conv}
\end{equation}

Since $f$ satisfies the PL condition,
\begin{equation}
    \frac{1}{2}\|\nabla f(\param_\indexopt)\|^2
    \ge \mu w_\indexopt .
    \notag
\end{equation}
Using the second equation of \eqref{relaxed_sav_scheme_for_proof}, this can be written as
\begin{equation}
    \frac{1}{2\alpha_\indexopt^2}
    \|\A_\indexopt(\paramdiff)\|^2
    \ge \mu w_\indexopt .
    \label{eq:pl_conv_1}
\end{equation}
Moreover, by \eqref{operator_assumption},
\begin{equation}
    \|\A_\indexopt(\paramdiff)\|
    \le \|\A_\indexopt\|\,\|\paramdiff\| 
    \le (1+\stepsize M)\|\paramdiff\| .
\notag
\end{equation}
Combining this estimate with \eqref{eq:pl_conv_1} yields
\begin{equation}
    \frac{(1+\stepsize M)^2}{2\alpha_\indexopt^2}
    \|\paramdiff\|^2
    \ge \mu w_\indexopt .
    \notag
\end{equation}
Without loss of generality, we may assume that $w_\indexopt>0$ for all $\indexopt$, since otherwise the conclusion is immediate. Therefore,
\begin{equation}
    \frac{1}{\sqrt{w_\indexopt}}
    \ge
    \frac{\sqrt{2\mu}\,\alpha_\indexopt}{(1+\stepsize M)\|\paramdiff\|} .
    \label{eq:pl_conv_2}
\end{equation}
By the monotonicity of $f$ proved in part (i), we have $w_\indexopt\ge w_{\indexopt+1}$. Therefore, 
it follows that
\begin{align*}
    \sqrt{w_\indexopt}-\sqrt{w_{\indexopt+1}}
    &=
    \frac{w_\indexopt-w_{\indexopt+1}}{\sqrt{w_\indexopt}+\sqrt{w_{\indexopt+1}}} 
    \ge
    \frac{w_\indexopt-w_{\indexopt+1}}{2\sqrt{w_\indexopt}} \\
    &\ge
    \frac{1}{4\alpha_\indexopt\sqrt{w_\indexopt}}\|\paramdiff\|^2 
    \ge
    \frac{\sqrt{2\mu}}{4(1+\stepsize M)}\|\paramdiff\| ,
\end{align*}
where we used \eqref{eq:L_smooth_conv} and \eqref{eq:pl_conv_2} in the last two inequalities.
Equivalently,
\begin{equation}
    \|\paramdiff\|
    \le
    \frac{4(1+\stepsize M)}{\sqrt{2\mu}}
    \left(\sqrt{w_\indexopt}-\sqrt{w_{\indexopt+1}}\right).
    \notag
\end{equation}
Summing this inequality from $n=0$ to $N$ yields
\begin{equation}
    \sum_{n=0}^{N}\|\paramdiff\|
    \le
    \frac{4(1+\stepsize M)}{\sqrt{2\mu}}
    \left(\sqrt{w_0}-\sqrt{w_{N+1}}\right)
    \le
    \frac{4(1+\stepsize M)}{\sqrt{2\mu}}\sqrt{w_0} .
    \notag
\end{equation}
Letting $N\to\infty$, we obtain
\begin{equation}
    \sum_{n=0}^{\infty}\|\paramdiff\|<\infty .
    \label{eq:sum_bound}
\end{equation}
Therefore, $\{\param_\indexopt\}_\indexopt$ is a Cauchy sequence.
Since $\R^d$ is complete,
there exists $\param_\infty\in\R^d$ such that $\lim_{\indexopt\to\infty}\param_\indexopt=\param_\infty$.
By part (i), we have
$
\|\nabla f(\param_\indexopt)\| \to 0
$ as $\indexopt \to \infty$.
Since $\nabla f$ is continuous, this implies that
$\nabla f(\param_\infty)=\bm{0}$.
Hence, $\param_\infty$ is a stationary point. Under the PL condition, every stationary point is a global minimizer, which completes the proof.

\paragraph{Proof of (iii)}
By the $L$-smoothness of $f$, \eqref{Lsmooth_eq} and $\alpha_\indexopt\le1/L$ established in \eqref{eq:alpha_f_monotonic}, we have
\begin{equation}
    w_{\indexopt+1}-w_\indexopt
    \le
    -\frac{\alpha_\indexopt}{2}
    \|\nabla f(\param_\indexopt)\|_{\A_\indexopt^{-1}}^2.
    \notag
\end{equation}
By \eqref{grad_val_from_assumption}, we see that
\begin{equation}
    w_{\indexopt+1}-w_\indexopt
    \le
    -\frac{\alpha_\indexopt}{2(1+\stepsize M)}
    \|\nabla f(\param_\indexopt)\|^2 .
    \notag
\end{equation}
Applying the PL condition and \eqref{alpha_val} yields
\begin{equation*}
    w_{\indexopt+1}-w_\indexopt
    \le
    -\frac{\mu\alpha_\indexopt}{1+\stepsize M}w_\indexopt 
    \le
    -\frac{\mu\tilr_\ast\stepsize}{(1+\stepsize M)\sqrt{f(\param_0)+C}}
    w_\indexopt .
\end{equation*}
Therefore,
\begin{equation}
    w_{\indexopt+1}
    \le
    \left(
    1-\frac{\mu\tilr_\ast\stepsize}{(1+\stepsize M)\sqrt{f(\param_0)+C}}
    \right)w_\indexopt .
    \notag
\end{equation}
Define
\begin{equation}
    q\coloneqq
    1-\frac{\mu\tilr_\ast\stepsize}{(1+\stepsize M)\sqrt{f(\param_0)+C}} .
    \notag
\end{equation}
To show that $0<q<1$, it suffices to verify that
\begin{equation}
    \frac{\mu\tilr_\ast\stepsize}{(1+\stepsize M)\sqrt{f(\param_0)+C}}<1 .
    \label{eq:converge_rate_condition}
\end{equation}

A standard consequence of $L$-smoothness is
\begin{equation}
    \|\nabla f(\param)\|^2
    \le
    2L\bigl(f(\param)-f_\ast\bigr).
    \notag
\end{equation}
Combining this estimate with the PL condition yields $\mu\le L$. 
Moreover, by definition, $\tilde{r}_* \leq r_* \leq r_0 = \sqrt{f(\param_0)+C}$.
Therefore, 
if the step size satisfies
\[
    \stepsize\le
    \frac{\sqrt{\delta}}{L\sqrt{f(\param_0)+C}} \leq \frac{1}{L},
\]
the condition \eqref{eq:converge_rate_condition} is satisfied as follows:
\begin{equation}
    \frac{\mu\tilr_\ast\stepsize}{(1+\stepsize M)\sqrt{f(\param_0)+C}}\leq \frac{1}{1+\stepsize M} < 1 .
\end{equation}
This proves the claimed linear convergence rate.

\end{proof}

\subsection{Result for the Proposed Method (N-RSAV)}

The following lemma 
states a fundamental property of the Nystr\"om approximation.
\begin{lemma}[\cite{nystrom_lemma}]
\label{lem:nystrom_order}
Let $\bm{X}\in\R^{d\times d}$ be a positive semidefinite matrix, and let $\tilde{\bm{X}}\in\R^{d\times d}$ be its Nystr\"om approximation. Then
\begin{equation}
    \bm{0}\preceq \tilde{\bm{X}} \preceq \bm{X}. \notag
\end{equation}
\end{lemma}

We now obtain the following convergence result for the proposed N-RSAV method.

\begin{corollary}
\label{col:proposed_converge_convex}
Assume that $f$ is convex, $L$-smooth, and twice differentiable. Let $C_1>0$ be the constant in Lemma~\ref{thm:positive_lower_bound}. Suppose that the step size satisfies
\[
    \stepsize
    \le
    \min\left\{C_1,\frac{\sqrt{\delta}}{L\sqrt{f(\param_0)+C}}\right\} .
\]
Then the N-RSAV scheme with the operator $\operator_\indexopt=\tilde{\bm{H}}_\indexopt$ satisfies the same conclusions of Theorem~\ref{thm:rsav_convergence_analysis} with  $M= L$.
\end{corollary}

\begin{remark}
The convexity assumption in Corollary~\ref{col:proposed_converge_convex} is imposed to ensure that the Hessian $\bm{H}_f(\param_\indexopt)$ is positive semidefinite, which is required in order to apply Lemma~\ref{lem:nystrom_order}. 
In neural network training, however, the Hessian is generally not  positive semidefinite. 
Nevertheless, the empirical results in \cite{nn_hessian_spectrum_1} suggest that most negative eigenvalues of the Hessian disappear after the initial stage of training.
\end{remark}

\begin{proof}
We verify that the operator $\operator_\indexopt=\tilde{\bm{H}}_\indexopt$ used in the proposed method satisfies the assumption of Theorem~\ref{thm:rsav_convergence_analysis}, namely, that it is positive semidefinite and uniformly bounded above.
Since $f$ is convex, its Hessian $\bm{H}_f(\param_\indexopt)$ is positive semidefinite. Therefore, Lemma~\ref{lem:nystrom_order} implies that
\begin{equation}
    \bm{0}\preceq \tilde{\bm{H}}_\indexopt
    \preceq \Hessian(\param_\indexopt).
    \notag
\end{equation}
Moreover, since $f$ is $L$-smooth and twice differentiable, 
\begin{equation}
    \Hessian(\param_\indexopt)\preceq L\I^{(d)} .
    \notag
\end{equation}
Consequently,
\begin{equation}
    \bm{0}\preceq \tilde{\bm{H}}_\indexopt
    \preceq L\I^{(d)} .
    \notag
\end{equation}
Thus, the assumptions of Theorem~\ref{thm:rsav_convergence_analysis} are satisfied with $M=L$, and the conclusion follows immediately.
\end{proof}

\section{Numerical experiments}
\label{sec:numerics}

In the previous section, we established convergence properties of the proposed N-RSAV method. 
In this section, we investigate its practical performance through numerical experiments. 
In particular, we examine whether the proposed low-rank Hessian approximation can achieve convergence behavior comparable to that obtained using the exact Hessian while reducing the computational cost. 
We also assess the effectiveness of the adaptive Hessian reuse strategy employed in AN-RSAV. 
To this end, we consider both a convex quadratic optimization problem and a practical nonconvex optimization problem arising from the training of physics-informed neural networks (PINNs).

We implemented our method by using Python 3.9.13 and PyTorch 2.8.0+cu128. 
The experiments were conducted on a machine running Ubuntu 20.04.6 LTS, and all experiments were run on four NVIDIA A100 PCIe 40GB GPUs.
In all experiments, unless otherwise specified, we set the constant in the scalar auxiliary variable to $C = 1$, 
$\eta $ in \eqref{eq:cond} to $\eta = 0.99$,
and the hyperparameter in the RSAV operator to $\lambda = 0$. For AN-RSAV, we set the hyperparameters to $\recomputefactor = 1.5$.

\subsection{Experiment 1: Convex quadratic problem}

We first consider a convex optimization problem.
For this problem, we check the effect of incorporating Hessian information.
In particular, we consider an ill-conditioned problem whose Hessian exhibits an effectively low-rank structure, which is expected to be favorable for the proposed Nystr\"om approximation.

We consider the following quadratic objective function:
\begin{equation}
f(\param) = \frac{1}{2}\param^\top \bm{H} \param\quad \param \in \mathbb{R}^d. \notag
\end{equation}
This function is convex, and its minimizer is given by $\param^\ast = (0, \ldots, 0)^\top \in \R^d$, with $f(\param^\ast) = 0$. In this experiment, we set $d = 400$.
The Hessian $\bm{H}\in\R^{d \times d}$ is 
constructed so that its eigenvalue spectrum follows that shown in \cref{fig:hessian_spectrum}.
This yields an ill-conditioned problem whose Hessian has an effectively low-rank structure.

\begin{figure}[htbp]
    \centering
    \includegraphics[width=0.6\textwidth]{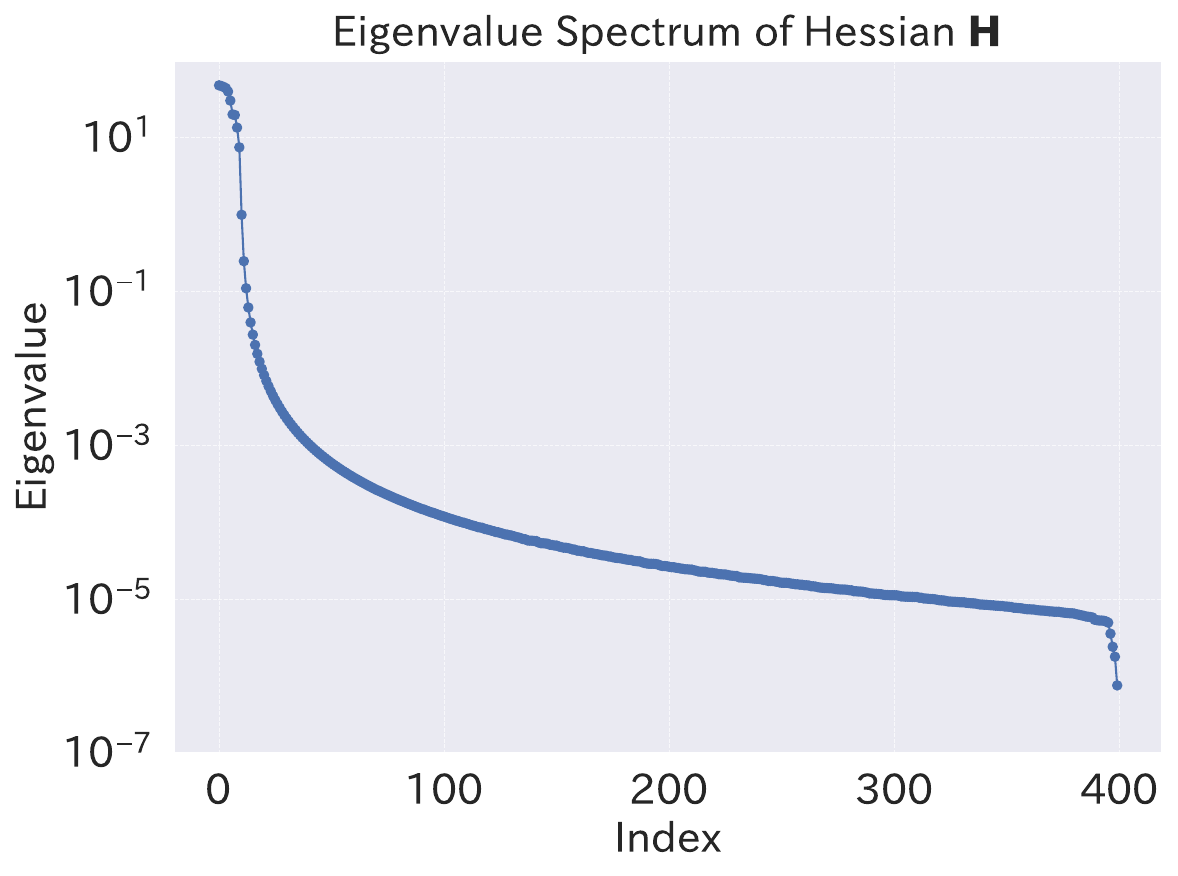}
    \hfill
    \caption{Eigenvalue spectrum of the Hessian $\bm{H}\in \R^{400\times400}$ in Experiment~1. 
    The eigenvalues consist of $\lambda_i^{\mathrm{low}} = i^{-2}$ $(i = 1, \ldots, 390)$ and 10 outlier eigenvalues $\lambda_i^{\mathrm{high}}$ $(i = 1, \ldots, 10)$ independently sampled from the uniform distribution on $[1,10]$. The condition number of $H$ is $2.3016 \times 10^7$.}
    \label{fig:hessian_spectrum}
\end{figure}

We compare the proposed methods, namely, N-RSAV and AN-RSAV, with the conventional RSAV method and a variant that directly incorporates the Hessian into the operator, i.e., $\operator_\indexopt = \bm{H}$, which we refer to as H-RSAV. 
In H-RSAV, the linear system
 $(\bm{I}^{(d)}+\stepsize \bm{H})\bm{x} = \nabla f(\param_\indexopt)$ is computed by using HVPs and the conjugate gradient (CG) method~\cite{CG}.
 Since the effective numerical rank of the Hessian is known in this synthetic example,
the rank parameter for N-RSAV and AN-RSAV is set to $m = 15$.

Table~\ref{tab:quadratic_final_loss} shows the final objective values after $10,000$ iterations for different step sizes. The results indicate that, unless the step size is very small, the methods incorporating Hessian information (H-RSAV, N-RSAV, and AN-RSAV) achieve significantly faster convergence compared to RSAV, and exhibit comparable performance to each other.
However, the computational cost differs significantly among these methods. 
Figure~\ref{fig:quadratic loss} compares the evolution of the objective value with respect to computation time, 
where each method is evaluated using its best-performing step size $\stepsize$ among those tested in Table~\ref{tab:quadratic_final_loss}.
 The results show that N-RSAV, which reduces the computational cost associated with Hessian construction and linear system solves, is approximately $3\times$ faster than H-RSAV. 
 Furthermore, AN-RSAV, which adaptively reuses the approximate Hessian based on the energy deviation, achieves 
 an additional speedup of approximately $3\times$ compared to N-RSAV.

These results demonstrate that, 
when the Hessian admits an effectively low-rank structure,
the proposed methods, N-RSAV and AN-RSAV, 
achieve convergence behavior comparable to that obtained using the exact Hessian, while substantially reducing the computational cost.

\begin{table}[htbp]
\centering
\caption{Final objective values after $10,000$ iterations for different step sizes in Experiment~1.
}
\label{tab:quadratic_final_loss}
\begin{tabular}{c|cccc}
 $\stepsize$ & RSAV & H-RSAV & N-RSAV & AN-RSAV  \\
\hline
0.1  & $0.01391$ & $0.008516$ &$0.008517$ & $0.008509$ \\
1  & $0.01242$ & $0.001830$ & $0.001825$& $0.001824$ \\
10  & $1.05395$& $1.04667\times 10^{-4}$& $1.04715\times 10^{-4}$& $1.04106 \times 10^{-4}$  \\
20 &   $0.69953$  &   $1.62397\times 10^{-5}$  & $1.63579\times 10^{-5}$ &  $1.61991\times 10^{-5}$ \\
30  &  $0.16114$  & $3.23008\times 10^{-6}$&$3.21628 \times 10^{-6}$&$3.20370\times 10^{-6}$ \\
\end{tabular}
\end{table}

\begin{figure}[htbp]
    \centering
    \includegraphics[width=0.6\textwidth]{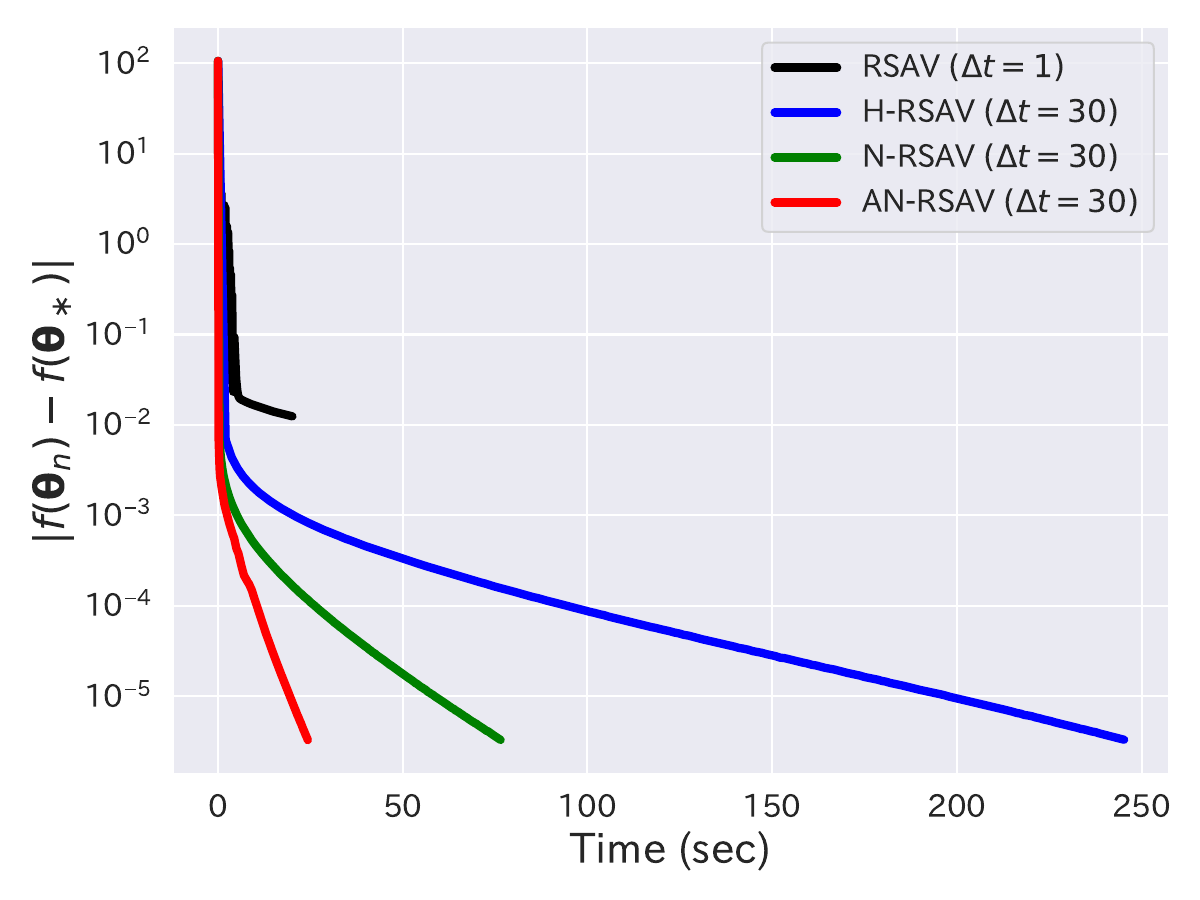}
    \hfill
    \caption{
    Evolution of the objective value with respect to computation time in Experiment~1.
    Each method is evaluated using its best-performing step size among those tested in Table~\ref{tab:quadratic_final_loss}.
    The total computation time for $10,000$ iterations is approximately 20 seconds for RSAV, 246 seconds for H-RSAV, 76 seconds for N-RSAV, and 24 seconds for AN-RSAV.}
    \label{fig:quadratic loss}
\end{figure}

\subsection{Experiment 2: PINNs training for the one-dimensional Burgers equation}

We next apply the proposed methods to the optimization problem arising in physics-informed neural networks (PINNs)~\cite{pinns}, a neural-network-based framework for solving partial differential equations.

The objective function in PINNs is generally nonconvex, and 
the Hessian of the loss function is often highly ill-conditioned~\cite{pinns_ill_condition}.
As a result, 
first-order methods and RSAV, whose search direction coincides with the negative gradient direction, may suffer from slow convergence.
Since the proposed methods incorporate curvature information through Hessian approximations, they are expected to improve optimization efficiency.
Furthermore, 
the Hessian in PINNs is known to exhibit an effectively low-rank structure~\cite{pinns_ill_condition},
which suggests that the proposed methods can be effective even when the rank parameter $m$ is small.

We consider the following one-dimensional time-dependent viscous Burgers equation:
\begin{equation}
    \left\{ \ 
 \begin{aligned}
    &  u_t + uu_x - \nu u_{xx} = 0 \quad x \in (-1,1), \ t\in (0, 1),  \\
    & u(x, 0) = -\sin(\pi x),  \\
    & u(-1, t) =  u(1,t) = 0.
 \end{aligned}
 \right. \label{eq:Burgers_eq}
\end{equation}
This problem is widely used as a benchmark for evaluating PINNs because the solution develops sharp gradients when the viscosity is small.
The viscosity parameter is set to $\nu = 0.01/\pi \approx0.00318$.
For this parameter value, the solution becomes steep, making the optimization problem more challenging.

We employ a fully connected neural network with 9 hidden layers and  20 neurons per layer, resulting in 
$d=3,441$ trainable parameters. The activation function is chosen as $\tanh$, and the parameters are initialized using Xavier initialization~\cite{Xavier_init}.
For the training data in the 
loss function, we use points sampled from a uniform distribution: $10,000$ collocation points in the space-time domain $(-1,1)\times (0,1)$ for the PDE residual loss, 200 points for the initial condition loss, and 200 points for the boundary condition loss.
As a reference solution for evaluating the accuracy of the PINNs approximation, we use a numerical solution obtained by discretizing the spatial domain with a Chebyshev pseudospectral method~\cite{spectral_method} and the temporal domain with the Radau IIA method~\cite{radauIIA}, which is a type of implicit Runge–Kutta method. The test loss is defined as the L2 absolute error between the approximate solution and the reference solution.

Since the loss function is nonconvex, 
we follow the strategy discussed in subsection~\ref{subsec:remarks}:
first use Adam to approach a local minimizer before switching to RSAV-based methods.

In addition, since the proposed methods can be classified as quasi-second-order optimization methods, we also compare them with L-BFGS~\cite{lbfgs}, a quasi-Newton method.
We compare the following five optimization methods:
\begin{itemize}[leftmargin=2em]
    \item Adam (23,000 iter)
    \item Adam (3,000 iter) + L-BFGS (20,000 iter)
    \item Adam (3,000 iter) + RSAV (20,000 iter)
    \item Adam (3,000 iter) + N-RSAV (20,000 iter)
    \item Adam (3,000 iter) + AN-RSAV (20,000 iter)
\end{itemize}
The first method uses Adam alone. The remaining methods perform 3,000 iterations of Adam in the initial stage of optimization, followed by switching to L-BFGS, RSAV, N-RSAV, or AN-RSAV, respectively.
For the hyperparameter settings, the step size of Adam is set to 
$1.0 \times 10^{-3}$. The step size of RSAV is $\stepsize = 1.0 \times 10^{-3}$ while for N-RSAV and AN-RSAV, the step size is 
$\stepsize=8.0$ and the rank parameter is $m=30$.

\cref{fig:pinns_burgers_test_loss} 
compares the evolution of the test loss with respect to both the number of iterations and computation time.
The proposed methods N-RSAV and AN-RSAV achieve the lowest final test losses among all methods considered, outperforming both Adam and Adam combined with L-BFGS.
It is also observed that L-BFGS stagnates 
shortly after the switch from Adam (approximately 300 iterations after the switch).
This behavior is likely due to the difficulties in selecting a suitable positive step size under the strong Wolfe conditions~\cite{back}, which are used to ensure descent directions.
In contrast, the proposed methods, N-RSAV and AN-RSAV, continue to reduce the test loss without stagnation, demonstrating their effectiveness.
Furthermore, AN-RSAV is approximately $3\times$ faster than N-RSAV in terms of computation time required for the same number of iterations. 
This implies the effectiveness of the adaptive Hessian reuse strategy based on the energy deviation.

\cref{fig:pinns_burgers_error_colormap} further compares the pointwise absolute error over the space-time domain.
The solution obtained by AN-RSAV exhibits smaller errors across most of the domain than that obtained by L-BFGS, strongly indicating the improved optimization performance of AN-RSAV.

\begin{figure}[htbp]
    \centering
    \begin{subfigure}{0.48\textwidth}
        \centering
        \includegraphics[width=\textwidth]{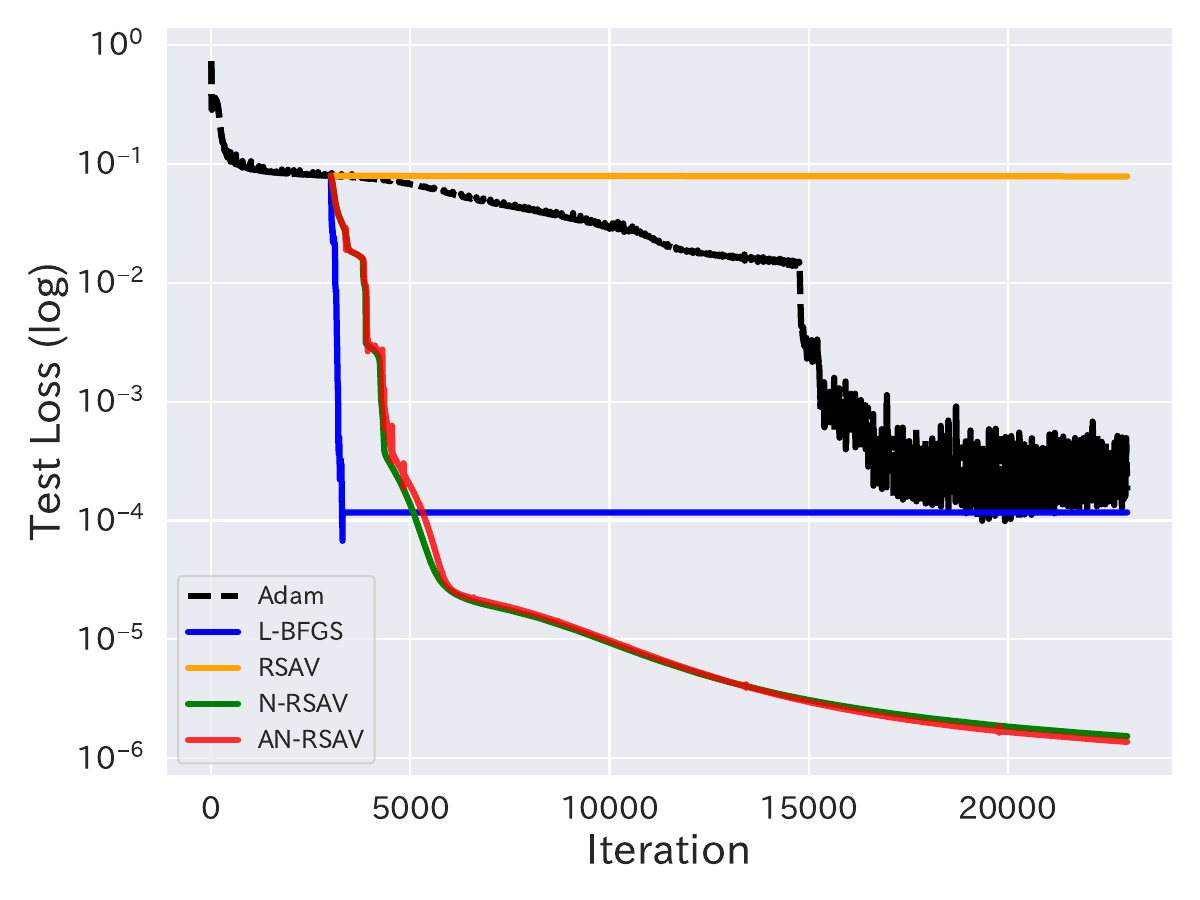}
        \caption{Test loss versus iterations}
    \end{subfigure}
    \hfill
    \begin{subfigure}{0.48\textwidth}
        \centering
        \includegraphics[width=\textwidth]{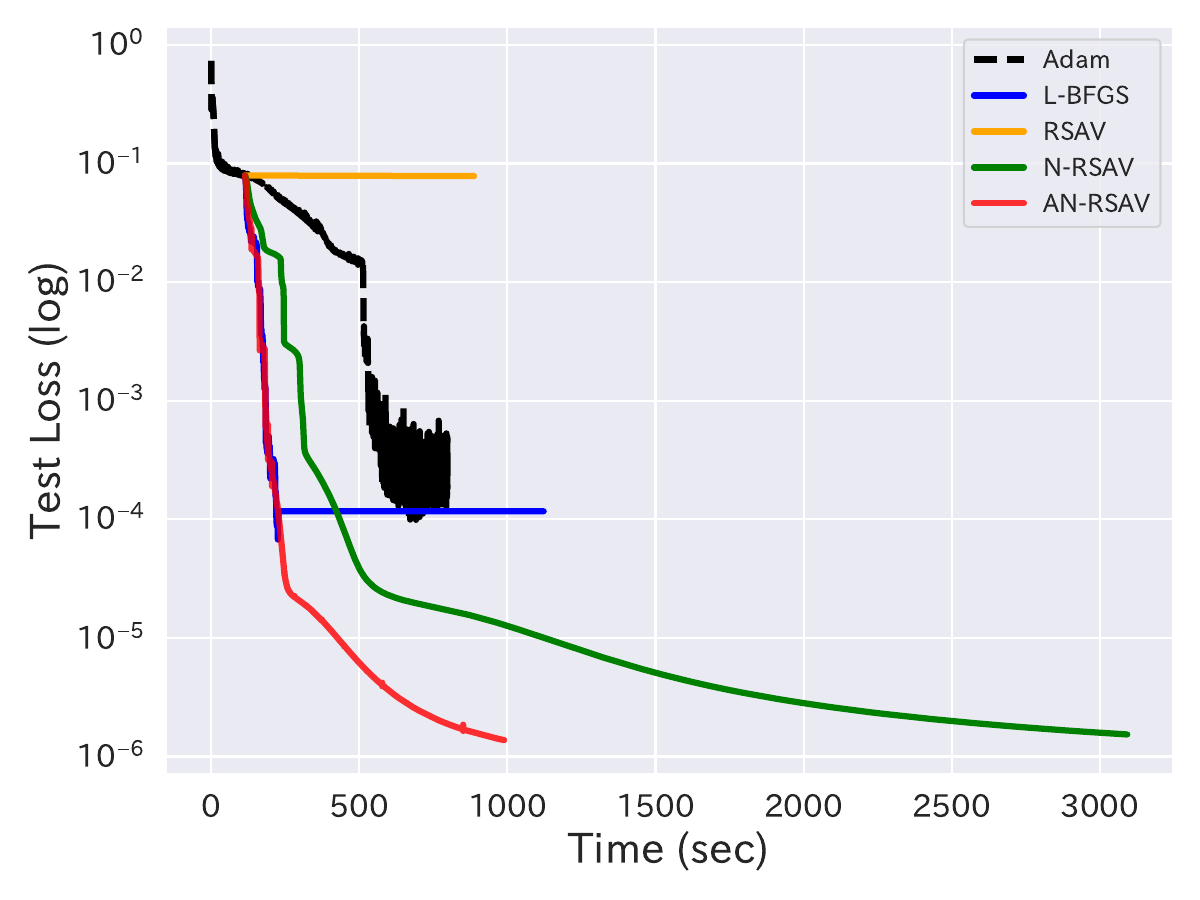}
        \caption{Test loss versus computation time}
    \end{subfigure}
    \caption{
    Comparison of the test loss for PINNs training of the one-dimensional Burgers equation. Left: test loss versus iterations. Right: test loss versus computation time.
    The step size of Adam is set to
    $1.0\times 10^{-3}$, the initial step size of L-BFGS is 1.0, and the step size of RSAV is 
    $\stepsize =1.0 \times 10^{-3}$.
For the proposed methods N-RSAV and AN-RSAV, the hyperparameters are set to $\stepsize = 8.0$ and $ m=30$. 
The final test loss is $1.53338\times 10^{-6}$
 for N-RSAV and $1.37462\times 10^{-6}$ for AN-RSAV.}
    \label{fig:pinns_burgers_test_loss}
\end{figure}

\begin{figure}[htbp]
    \centering
    \includegraphics[width=0.9\textwidth]{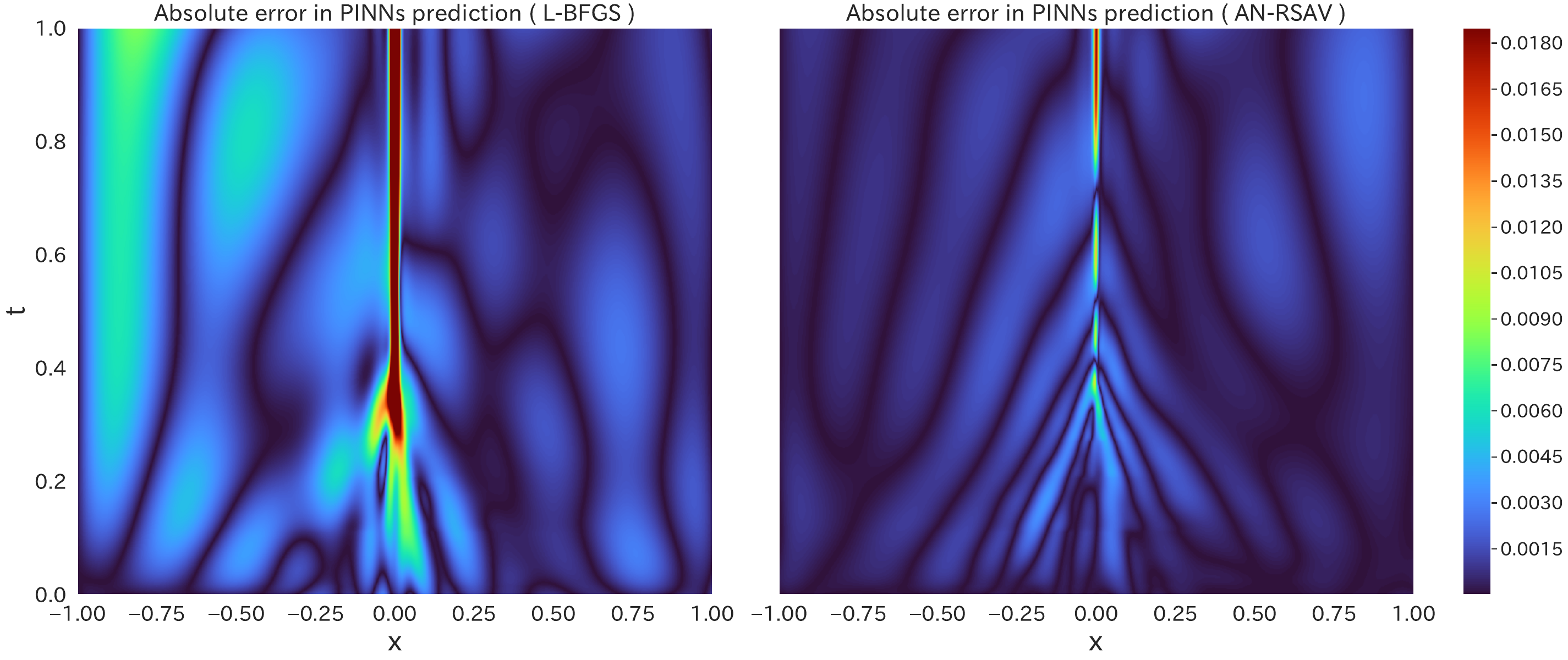}
    \hfill
    \caption{
    Pointwise absolute error between the PINNs approximation and the reference solution over the space-time domain. Left: Adam + L-BFGS. Right: Adam + AN-RSAV. The proposed method yields smaller errors over most of the domain, indicating a more accurate approximation of the Burgers solution.
}
    \label{fig:pinns_burgers_error_colormap}
\end{figure}

\section{Conclusion}
\label{sec:conclusion}
In this paper, we proposed a Nyström-enhanced RSAV method for continuous optimization, which incorporates approximate Hessian information into the RSAV framework. By combining a randomized low-rank approximation with eigenvalue truncation, the proposed method effectively exploits curvature information while preserving the modified dissipation law.
In addition, we introduced an adaptive strategy that reuses the approximate Hessian based on the deviation between the original and modified energies. This mechanism significantly reduces the computational cost 
of
repeatedly constructing the approximate Hessian.
We also established convergence analysis for the RSAV scheme with a general positive semidefinite operator $\operator_\indexopt$ under a uniform upper bound assumption in the Loewner order. 
The analysis proves that the objective function values decrease monotonically.
Under the PL condition, it further establishes convergence of the iterates to a global minimizer and a linear convergence rate for the objective gap. 
In the convex setting, these results imply the corresponding convergence guarantees for the proposed N-RSAV method. 
Moreover, the generality of the analysis may provide a useful theoretical basis for future studies that improve RSAV schemes through the design of problem-dependent operators $\operator_\indexopt$.

Through numerical experiments on both convex quadratic problems and PINNs training for the Burgers equation, we demonstrated that the proposed methods, N-RSAV and AN-RSAV, achieve significantly faster convergence than RSAV. 
In the quadratic optimization problem, N-RSAV was shown to reduce the computational cost per iteration compared with the case where the exact Hessian is directly used as the operator $\operator_\indexopt$.
Furthermore, AN-RSAV maintains performance comparable to that of N-RSAV, while further reducing computational cost by reusing the approximate Hessian.
Moreover, the results on PINNs training suggest that the proposed methods provide an efficient quasi-second-order optimization approach for ill-conditioned problems with an effectively low-rank structure.

A limitation of the proposed methods lies in the choice of the rank parameter $m$, which is treated as a hyperparameter. If $m$ is chosen to be too small, the optimization may fail due to insufficient approximation accuracy. On the other hand, increasing $m$ improves the Hessian approximation accuracy but increases the computational cost per iteration. Therefore, selecting an appropriate value of $m$ is crucial.
Addressing this issue remains an important direction for future work, and approaches such as fast rank estimation methods~\cite{effecive_rank_inference1,effecive_rank_inference2} may provide useful insights.

\section*{Acknowledgments}
This work is supported by JSPS KAKENHI Grant Numbers 21K18301, 24K02951, 24K00540, 25H00449, and 25K21806, MEXT STAR-E NEXT Project (Japan Grant Number JPJ013735), and JST SPRING Grant Number JPMJSP2138. 

\bibliographystyle{siamplain}
\bibliography{refs}
\end{document}